%% file: main.tex
\documentclass{amsart}

\pdfoutput=1

\usepackage{amsmath}
\usepackage{amssymb}
\usepackage{bbm}
\usepackage{enumerate}
\usepackage{booktabs} 
\usepackage{appendix}
\usepackage{mathrsfs}
\usepackage{yhmath}
\usepackage{hyperref}
\usepackage{graphicx, color}
\hypersetup{
  pdftitle={Brownian bridge expansions, L{\'e}vy area and the Riemann zeta function},
  pdfauthor={James Foster and Karen Habermann},
}

\usepackage[T1]{fontenc}
\usepackage[utf8]{inputenc}

\usepackage{xcolor}

\numberwithin{equation}{section}

\newcommand{\R}{\mathbb{R}}
\newcommand{\N}{\mathbb{N}}

\newcommand{\C}{\mathbb{C}}
\newcommand{\E}{\mathbb{E}}
\newcommand{\pr}{\mathbb{P}}

\newcommand{\e}{\operatorname{e}}
\newcommand{\dd}{\,{\mathrm d}}
\newcommand{\db}{{\mathrm d}}

\newcommand{\im}{\operatorname{i}}
\newcommand{\ind}{\mathbbm{1}}

\newcommand{\eps}{\varepsilon}

\newcommand{\cov}{\operatorname{cov}}

\DeclareFontFamily{U}{mathx}{\hyphenchar\font45}
\DeclareFontShape{U}{mathx}{m}{n}{
      <5> <6> <7> <8> <9> <10>
      <10.95> <12> <14.4> <17.28> <20.74> <24.88>
      mathx10
      }{}
\DeclareSymbolFont{mathx}{U}{mathx}{m}{n}
\DeclareFontSubstitution{U}{mathx}{m}{n}
\DeclareMathAccent{\widebar}{0}{mathx}{"73}

\newcommand{\m}{\hspace{0.25mm}}

\newtheorem{lemma}{Lemma}[section]

\newtheorem{propn}[lemma]{Proposition}
\newtheorem{thm}[lemma]{Theorem}
\newtheorem{cor}[lemma]{Corollary}
\newtheorem{defn}[lemma]{Definition}
\newtheorem{remark0}[lemma]{Remark}
\newtheorem{eg0}[lemma]{Example}

\author[J. Foster]{James Foster}
\address{James Foster, Department of Mathematical Sciences, University of Bath,
  Bath, BA2 7AY, United Kingdom.}
\email{jmf68@bath.ac.uk}
\thanks{James Foster was supported by the Department of Mathematical Sciences at the University of Bath and the DataSig programme under the EPSRC grant EP/S026347/1.}
\author[K. Habermann]{Karen Habermann}
\address{Karen Habermann, Department of Statistics, University of
  Warwick, Coventry, CV4 7AL, United Kingdom.}
\email{karen.habermann@warwick.ac.uk}

\title[Brownian bridge expansions, L{\'e}vy area and the Riemann zeta
  function]{Brownian bridge expansions for L{\'e}vy area approximations and
  particular values of the Riemann zeta function}
\begin{document}
\begin{abstract}
We study approximations for the L{\'e}vy area of Brownian motion which are based on the Fourier series expansion and a polynomial expansion of the associated Brownian bridge. Comparing the asymptotic convergence rates of the L{\'e}vy area approximations, we see that the approximation resulting from the polynomial expansion of the Brownian bridge is more accurate than the Kloeden--Platen--Wright approximation, whilst still only using independent normal random vectors. We then link the asymptotic convergence rates of these approximations to the limiting fluctuations for the corresponding series expansions of the Brownian bridge.
Moreover, and of interest in its own right, the analysis we use to identify the fluctuation processes for the Karhunen--Lo{\`e}ve and Fourier series expansions of the Brownian bridge is extended to give a stand-alone derivation of the values of the Riemann zeta function at even positive integers.
\end{abstract}

\maketitle
\thispagestyle{empty}
\section{Introduction}
\input{introduction}

\section{Series expansions for the Brownian bridge}
\label{sect:expansions}
\input{expansions}

\section{Particular values of the Riemann zeta function}\label{sect:zeta}
\input{zeta}

\section{Fluctuations for the trigonometric expansions of the Brownian bridge}
\label{sect:fluct}
We first prove Theorem~\ref{thm:sinflucutation} and Corollary~\ref{cor1} which we use to determine
the pointwise limits for the covariance functions of the fluctuation processes for the Karhunen--Lo{\`e}ve expansion and of the fluctuation processes for the Fourier series expansion, and then we deduce Theorem~\ref{thm:fluctuations}.

\subsection{Fluctuations for the Karhunen--Lo{\`e}ve expansion}
\label{sect:fluct1}
\input{fluctuations1}

\subsection{Fluctuations for the Fourier series expansion}
\label{sect:fluct2}
\input{fluctuations2}

\section{Approximations of Brownian L\'{e}vy area}\label{sect:levyarea}
\input{levyarea}
\newpage
\appendix

\section{Summarising tables}
\label{appendix}
\input{appendix}
\bibliographystyle{plain}
\bibliography{references}

\end{document}

%% file: introduction.tex
One of the well-known applications for expansions of the Brownian bridge is the strong or $L^2(\mathbb{P})$ approximation of stochastic integrals.
Most notably, the second iterated integrals of Brownian motion are required by high order strong numerical methods for general stochastic differential equations (SDEs), as discussed in~\cite{CameronClark, kloedenplaten, StrongSRK}.
Due to integration by parts, such integrals can be expressed in terms of the increment and L\'{e}vy area of Brownian motion.
The approximation of multidimensional L\'{e}vy area is well-studied, see \cite{Davie, Dickinson, Fosterthesis, GainesLyonsInt, GainesLyonsVar, KloedenPlatenWright, KuznetsovLevyArea2, RecentLevyArea, Wiktorsson}, with the majority of the algorithms proposed being based on a Fourier series expansion or the standard piecewise linear approximation of Brownian motion. Some alternatives include \cite{Davie, Fosterthesis, KuznetsovLevyArea2} which consider methods associated with a polynomial expansion of the Brownian bridge.\medskip

Since the advent of Multilevel Monte Carlo (MLMC), introduced by Giles in \cite{MLMC} and subsequently developed in \cite{WeakMLMC, WeakAntitheticMLMC, HighOrderMLMC, MilsteinMLMC, AntitheticMLMC}, L\'{e}vy area approximation has become less prominent in the literature. In particular, the antithetic MLMC method introduced by Giles and Szpruch in \cite{AntitheticMLMC} achieves the optimal complexity for the weak approximation of multidimensional SDEs without the need to generate Brownian L\'{e}vy area. That said, there are concrete applications where the simulation of L\'{e}vy area is beneficial, such as for sampling from non-log-concave distributions using It\^{o} diffusions.
For these sampling problems, high order strong convergence properties of the SDE solver lead to faster mixing properties of the resulting Markov chain Monte Carlo (MCMC) algorithm, see \cite{SRKforMCMC}.\medskip

In this paper, we compare the approximations of L\'{e}vy area based on the Fourier series expansion and on a polynomial expansion of the Brownian bridge. We particularly observe their convergence rates and link those to the fluctuation processes associated with the different expansions of the Brownian bridge. The fluctuation process for the polynomial expansion is studied in~\cite{semicircle}, and our study of the fluctuation process for the Fourier series expansion allows us, at the same time, to determine the fluctuation process for the Karhunen--Lo{\`e}ve expansion of the Brownian bridge.
As an attractive side result, we extend the required analysis to obtain a stand-alone derivation of the values of the Riemann zeta function at even positive integers. Throughout, we denote the positive integers by $\N$ and the non-negative integers by $\N_0$.

\medskip

Let us start by considering a Brownian bridge $(B_t)_{t\in [0,1]}$ in $\R$ with $B_0=B_1=0$. This is the unique continuous-time Gaussian process with mean zero and whose covariance function $K_B$ is given by, for $s,t\in[0,1]$,
\begin{equation}\label{cov_BB}
    K_B(s,t)=\min(s,t)-st\;.
\end{equation}
We are concerned with the following three expansions of the Brownian bridge.
The Karhunen--Lo{\`e}ve expansion of the Brownian bridge, see Lo{\`e}ve~\cite[p.~144]{loeve2},
is of the form, for $t\in[0,1]$,
\begin{equation}\label{eq:bridge_kl_expansion}
  B_t=\sum_{k=1}^\infty\frac{2\sin(k\pi t)}{k\pi}
  \int_0^1\cos(k\pi r)\dd B_r\;.
\end{equation}
The Fourier series expansion of the Brownian bridge, see Kloeden--Platen~\cite[p.~198]{kloedenplaten} or Kahane~\cite[Sect.~16.3]{kahane}, yields, for $t\in[0,1]$,
\begin{equation}\label{eq:bridge_fourier_expansion}
  B_t=\frac{1}{2}a_0+\sum_{k=1}^\infty\left(
    a_k\cos(2k\pi t)+b_k\sin(2k\pi t)
  \right)\;,
\end{equation}
where, for $k\in\N_0$, 
\begin{equation}\label{eq:fourier_coefficients}
  a_k=2\int_0^1 \cos(2k\pi r)B_r\dd r\quad\text{and}\quad
  b_k=2\int_0^1 \sin(2k\pi r)B_r\dd r\;.
\end{equation}
A polynomial expansion of the Brownian bridge in terms of the shifted Legendre polynomials $Q_k$ on the interval $[0,1]$ of degree $k$, see~\cite{foster,semicircle}, is given by, for $t\in[0,1]$,
\begin{equation}\label{eq:bridge_poly_expansion}
    B_t=\sum_{k=1}^\infty (2k+1)\m c_k \int_0^t Q_k(r)\dd r\;,
\end{equation}
where, for $k\in\N$, 
\begin{equation}\label{eq:poly_coefficients}
  c_k=\int_0^1 Q_k(r) \dd B_r\;.
\end{equation}
These expansions are summarised in Table~\ref{table:bridge_expansions} in Appendix~\ref{appendix} and they are discussed in more detail in Section~\ref{sect:expansions}. For an implementation of the corresponding approximations for Brownian motion as Chebfun examples into MATLAB, see Filip, Javeed and Trefethen~\cite{chebfun2} as well as Trefethen~\cite{chebfun1}.\medskip

We remark that the polynomial expansion~(\ref{eq:bridge_poly_expansion}) can be viewed as a Karhunen--Lo{\`e}ve expansion of the Brownian bridge with respect to the weight function $w$ on $(0,1)$ given by $w(t) = \frac{1}{t(1-t)}$. This approach is employed in~\cite{foster} to derive the expansion along with the standard optimality property of Karhunen--Lo{\`e}ve expansions. In this setting, the polynomial approximation of $(B_t)_{t\in [0,1]}$ is optimal among truncated series expansions in a weighted $L^2(\mathbb{P})$ sense corresponding to the non-constant weight function $w$.
To avoid confusion, we still adopt the convention throughout to reserve the term Karhunen--Lo{\`e}ve expansion for~(\ref{eq:bridge_kl_expansion}), whereas~(\ref{eq:bridge_poly_expansion}) will be referred to as the polynomial expansion.

\medskip

Before we investigate the approximations of L\'{e}vy area based on the different expansions of the Brownian bridge, we first analyse the fluctuations associated with the expansions. 
The fluctuation process for the polynomial expansion is studied and characterised in~\cite{semicircle}, and these results are recalled in Section~\ref{sect:poly}.
The fluctuation processes $(F_t^{N,1})_{t\in[0,1]}$ for the Karhunen--Lo{\`e}ve expansion and the fluctuation processes $(F_t^{N,2})_{t\in[0,1]}$ for the Fourier series expansion are defined as, for $N\in\N$,
\begin{equation}\label{KL_fluctuation}
    F_t^{N,1}=\sqrt{N}\left(B_t-\sum_{k=1}^N\frac{2\sin(k\pi t)}{k\pi}\int_0^1\cos(k\pi r)\dd B_r\right)\;,
\end{equation}
and
\begin{equation}\label{fourier_fluctuation}
    F_t^{N,2}=\sqrt{2N}\left(B_t-\frac{1}{2}a_0-\sum_{k=1}^N\left(a_k\cos(2k\pi t)+b_k\sin(2k\pi t)\right)\right)\;.
\end{equation}
The scaling by $\sqrt{2N}$ in the process $(F_t^{N,2})_{t\in[0,1]}$ is the natural scaling to use because increasing $N$ by one results in the subtraction of two additional Gaussian random variables.
We use $\E$ to denote the expectation with respect to Wiener measure $\pr$.

\begin{thm}\label{thm:fluctuations}
    The fluctuation processes $(F_t^{N,1})_{t\in[0,1]}$ for the Karhunen--Lo{\`e}ve expansion converge in finite dimensional distributions as $N\to\infty$ to the collection $(F_t^1)_{t\in[0,1]}$ of independent Gaussian random variables with mean zero and variance
    \begin{equation*}
        \E\left[\left(F_t^1\right)^2\m\right]=
        \begin{cases}
            \frac{1}{\pi^2} & \text{if }t\in(0,1)\\
            0 & \text{if } t=0\text{ or }t=1
        \end{cases}\;.
    \end{equation*}
    The fluctuation processes $(F_t^{N,2})_{t\in[0,1]}$ for the Fourier expansion converge in finite dimensional distributions as $N\to\infty$ to the collection $(F_t^2)_{t\in[0,1]}$ of zero-mean Gaussian random variables whose covariance structure is given by, for $s,t\in[0,1]$,
    \begin{equation*}
        \E\left[F_s^2 F_t^2\m\right]=
        \begin{cases}
            \frac{1}{\pi^2} & \text{if } s=t\text{ or } s,t\in\{0,1\}\\\
            0 & \text{otherwise}
        \end{cases}\;.
    \end{equation*}
\end{thm}
The difference between the fluctuation result for the Karhunen--Lo{\`e}ve expansion and the fluctuation result for the polynomial expansion, see~\cite[Theorem~1.6]{semicircle} or Section~\ref{sect:poly}, is that there the variances of the independent Gaussian random variables follow the semicircle $\frac{1}{\pi}\sqrt{t(1-t)}$ whereas here they are constant on $(0,1)$, see Figure~\ref{diag:fluctuation}. The limit fluctuations for the Fourier series expansion further exhibit endpoints which are correlated.\medskip

\begin{figure}[h]
\centering
\includegraphics[width=0.9\textwidth]{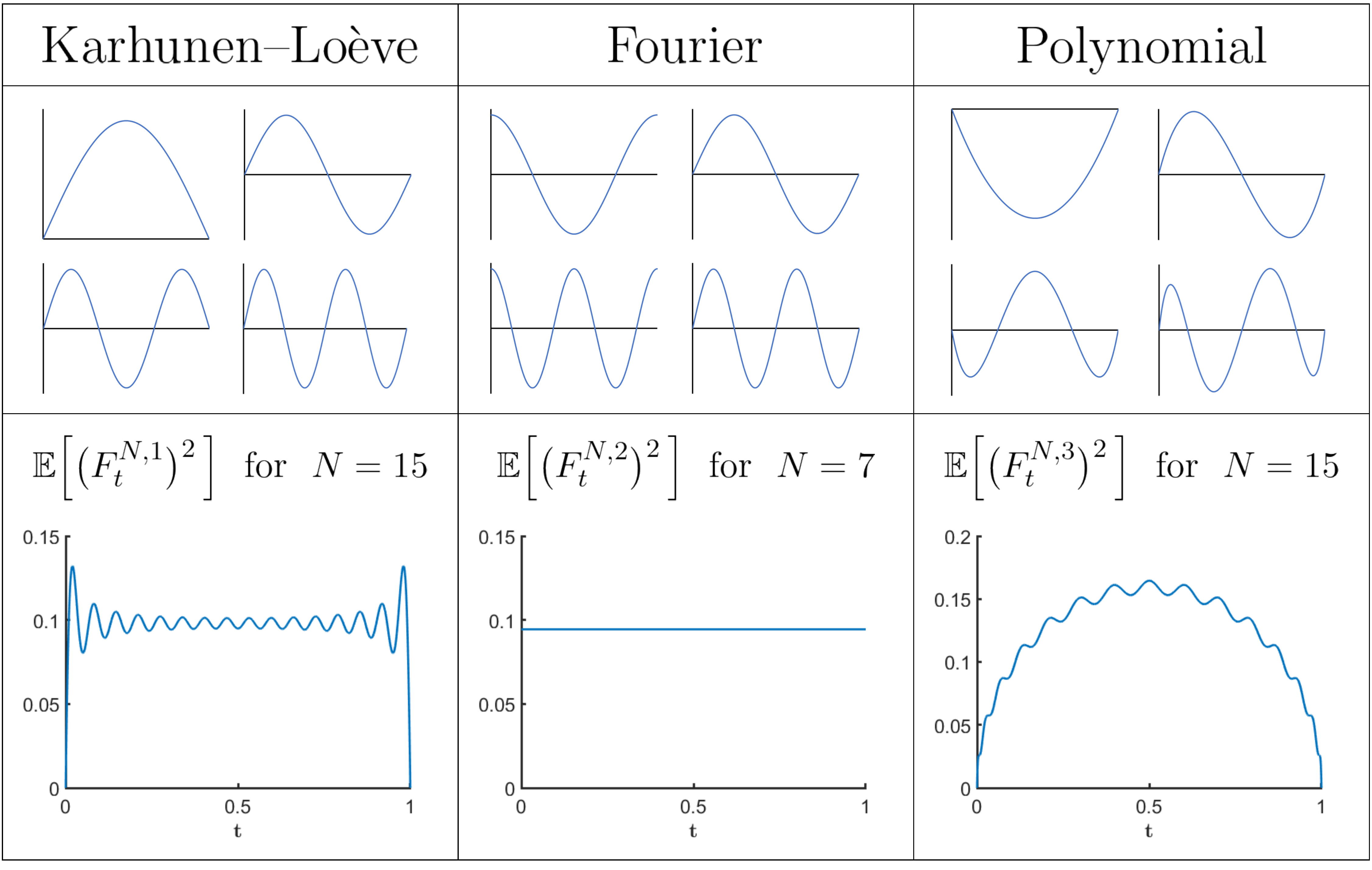}
\caption{Table showing basis functions and fluctuations for the Brownian bridge expansions.}
\label{diag:fluctuation}
\end{figure}

As pointed out in~\cite{semicircle}, the reason for considering convergence in finite dimensional distributions for the fluctuation processes is that the limit fluctuations neither have a realisation as processes in $C([0,1],\R)$, nor are they equivalent to measurable processes.\medskip

We prove Theorem~\ref{thm:fluctuations} by studying the covariance functions of the Gaussian processes $(F_t^{N,1})_{t\in[0,1]}$ and $(F_t^{N,2})_{t\in[0,1]}$ given in Lemma~\ref{lem:cov1} and Lemma~\ref{lem:cov2} in the limit $N\to\infty$. The key ingredient is the following limit theorem for sine functions, which we see concerns the pointwise convergence for the covariance function of $(F_t^{N,1})_{t\in[0,1]}$.
\begin{thm}\label{thm:sinflucutation}
    For all $s,t\in[0,1]$, we have
    \begin{equation*}
        \lim_{N\to\infty}N\left(\min(s,t)-st-\sum_{k=1}^N\frac{2\sin(k\pi s)\sin(k\pi t)}{k^2\pi^2}\right)=
        \begin{cases}
            \frac{1}{\pi^2} & \text{if } s=t\text{ and } t\in(0,1)\\
            0 & \text{otherwise}
        \end{cases}\;.
    \end{equation*}
\end{thm}
The above result serves as one of four base cases in the analysis performed in~\cite{HabermannRecent} of the asymptotic error arising when approximating the Green's function of a Sturm--Liouville problem through a truncation of its eigenfunction expansion. The work~\cite{HabermannRecent} offers a unifying view for Theorem~\ref{thm:sinflucutation} and~\cite[Theorem~1.5]{semicircle}.

\medskip

The proof of Theorem~\ref{thm:sinflucutation} is split into an on-diagonal and an off-diagonal argument. We start by proving the convergence on the diagonal away from its endpoints by establishing locally uniform convergence, which ensures continuity of the limit function, and by using a moment argument to identify the limit. As a consequence of the on-diagonal convergence, we obtain the next corollary which then implies the off-diagonal convergence in Theorem~\ref{thm:sinflucutation}.
\begin{cor}\label{cor1}
    For all $t\in(0,1)$, we have
    \begin{equation*}
        \lim_{N\to\infty} N\sum_{k=N+1}^\infty\frac{\cos(2k\pi t)}{k^2\pi^2}=0\;.
    \end{equation*}
\end{cor}
Moreover, and of interest in its own right, the moment analysis we use to prove the on-diagonal convergence in Theorem~\ref{thm:sinflucutation} leads to a stand-alone derivation of the result that the values of the Riemann zeta function $\zeta\colon\C\setminus\{1\}\to\C$ at even positive integers can be expressed in terms of the Bernoulli numbers $B_{2n}$ as, for $n\in\N$,
\begin{equation*}
  \zeta(2n)=(-1)^{n+1}\frac{\left(2\pi\right)^{2n}B_{2n}}{2(2n)!}\;,
\end{equation*}
see Borevich and Shafarevich~\cite{zeta_bernoulli}.
In particular, the identity
\begin{equation}\label{eq:basel}
    \sum_{k=1}^\infty \frac{1}{k^2} =\frac{\pi^2}{6}\;,
\end{equation}
that is, the resolution to the Basel problem posed by Mengoli~\cite{basel} is a consequence of our analysis and not a prerequisite for it.

\medskip

We turn our attention to studying approximations of second iterated integrals of Brownian motion.
For $d\geq 2$, let $(W_t)_{t\in[0,1]}$ denote a $d$-dimensional Brownian motion and let $(B_t)_{t\in[0,1]}$ given by $B_t=W_t-tW_1$ be its associated Brownian bridge in $\R^d$. We denote the independent components of $(W_t)_{t\in[0,1]}$ by $(W_t^{(i)})_{t\in[0,1]}$, for $i\in\{1,\dots,d\}$, and the components of $(B_t)_{t\in[0,1]}$ by $(B_t^{(i)})_{t\in[0,1]}$, which are also independent by construction.
We now focus on approximations of L\'{e}vy area.

\begin{defn}\label{def:levy_area} The L\'{e}vy area of the $d$-dimensional Brownian motion $W$ over the interval $[s,t]$ is the antisymmetric $d\times d$ matrix $A_{s,t}$ with the following entries, for $i,j\in\{1,\dots,d\}$,
\begin{equation*}
A_{s,t}^{(i,j)}  := \frac{1}{2}\left(\int_s^t \left(W_r^{(i)} - W_s^{(i)}\right)\dd W_r^{(j)} - \int_s^t \left(W_r^{(j)} - W_s^{(j)}\right)\dd W_r^{(i)}\right)\;.
\end{equation*}
\end{defn}
For an illustration of L\'{e}vy area for a two-dimensional Brownian motion, see Figure~\ref{diag:levyarea}.
\begin{remark0}\rm
Given the increment $W_t - W_s$ and the L\'{e}vy area $A_{s,t}\m$, we can recover the second iterated integrals of Brownian motion using integration by parts as, for $i,j\in\{1,\dots,d\}$ with $i\neq j$,
\begin{equation*}
    \int_s^t \left(W_r^{(i)} - W_s^{(i)}\right)\dd W_r^{(j)} = \frac{1}{2}\left(W_t^{(i)} - W_s^{(i)}\right)\left(W_t^{(j)} - W_s^{(j)}\right) + A_{s,t}^{(i,j)}\;.
\end{equation*}
\end{remark0}
\begin{figure}[ht]
\centering
\includegraphics[width=0.875\textwidth]{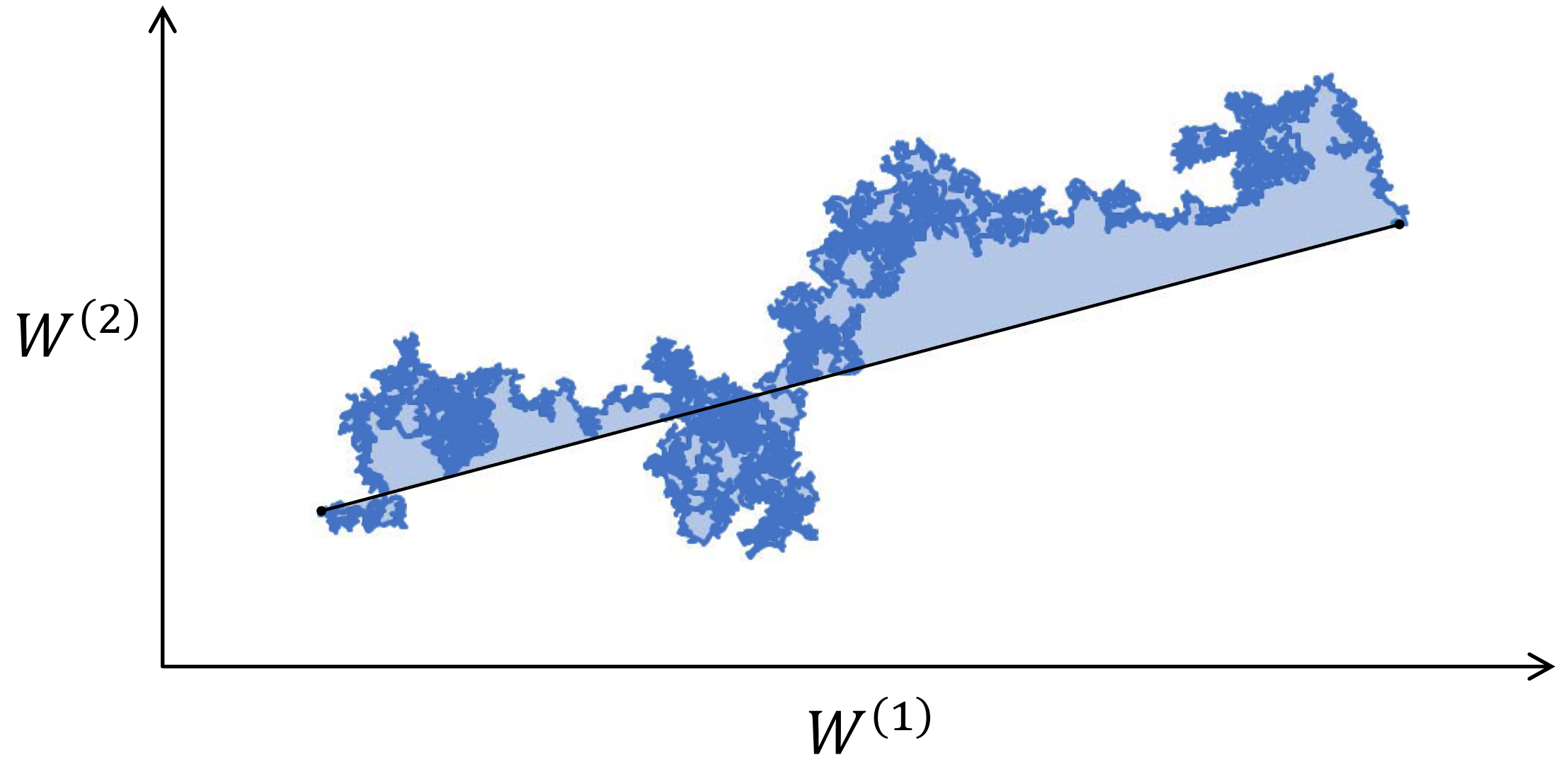}
\caption{L\'{e}vy area is the chordal area between independent Brownian motions.}
\label{diag:levyarea}
\end{figure}

We consider the sequences $\{a_k\}_{k\in\N_0}\m$, $\{b_k\}_{k\in\N}$ and $\{c_k\}_{k\in\N}$ of Gaussian random vectors, where the coordinate random variables $a_k^{(i)}$, $b_k^{(i)}$ and $c_k^{(i)}$ are defined for $i\in\{1,\dots,d\}$ by~(\ref{eq:fourier_coefficients}) and (\ref{eq:poly_coefficients}), respectively, in terms of the Brownian bridge $(B_t^{(i)})_{t\in[0,1]}$. Using the random coefficients arising from the Fourier series expansion~(\ref{eq:bridge_fourier_expansion}), we obtain the approximation of Brownian L\'{e}vy area proposed by Kloeden and Platen~\cite{kloedenplaten} and Milstein~\cite{MilsteinBook}. Further approximating terms so that only independent random coefficients are used yields the Kloeden--Platen--Wright approximation in~\cite{KloedenPlatenWright, MilsteinBook2, Wiktorsson}. Similarly, using the random coefficients from the polynomial expansion~(\ref{eq:bridge_poly_expansion}), we obtain the L\'{e}vy area approximation first proposed by Kuznetsov in \cite{KuznetsovLevyArea1}. These L\'{e}vy area approximations are summarised in Table~\ref{table:area_expansions} in Appendix~\ref{appendix} and have the following asymptotic convergence rates.

\begin{thm}[Asymptotic convergence rates of L\'{e}vy area approximations]\label{thm:asmyp_conv_rates_combined}
For $n\in\N$, we set $N=2n$ and define approximations $\widehat{A}_{n}$, $\widetilde{A}_{n}$ and $\widebar{A}_{2n}$ of the L\'{e}vy area $A_{0,1}$ by, for $i,j\in\{1,\dots,d\}$,
\begin{align}
	 \widehat{A}_{n}^{\m(i,j)} & := \frac{1}{2}\left(a_0^{(i)}W_1^{(j)} - W_1^{(i)}a_0^{(j)}\right) + \pi\sum_{k=1}^{n-1} k\left(a_{k}^{(i)}b_k^{(j)} - b_k^{(i)}a_{k}^{(j)}\right),\label{eq:levy_fourier_approx_intro}\\[3pt]
	 \widetilde{A}_{n}^{\m(i,j)} & := \pi\sum_{k=1}^{n-1} k\left(a_{k}^{(i)}\left(b_k^{(j)} - \frac{1}{k\pi}W_1^{(j)}\right) - \left(b_k^{(i)} - \frac{1}{k\pi}W_1^{(i)}\right)a_{k}^{(j)}\right),\label{eq:kpw_fourier_approx_intro}\\[3pt]
\widebar{A}_{2n}^{\m(i,j)} & := \frac{1}{2}\left(W_1^{(i)}c_1^{(j)} - c_1^{(i)}W_1^{(j)}\right) + \frac{1}{2}\sum_{k=1}^{2n-1}\left(c_k^{(i)}c_{k+1}^{(j)} - c_{k+1}^{(i)}c_k^{(j)}\right).\label{eq:poly_area_approx_intro}
\end{align}
Then $\widehat{A}_{n}$, $\widetilde{A}_{n}$ and $\widebar{A}_{2n}$ are antisymmetric $d \times d$ matrices and, for $i\neq j$ and as $N\to\infty$, we have
\begin{align*}
\E\bigg[\Big(A_{0,1}^{(i,j)} - \widehat{A}_{n}^{\m(i,j)}\Big)^2\m\bigg] & \sim \frac{1}{\pi^2}\bigg(\frac{1}{N}\bigg)\;,\\[3pt]
\E\bigg[\Big(A_{0,1}^{(i,j)} - \widetilde{A}_{n}^{\m(i,j)}\Big)^2\m\bigg] & \sim \frac{3}{\pi^2}\bigg(\frac{1}{N}\bigg)\;,\\[3pt]
\E\bigg[\Big(A_{0,1}^{(i,j)} - \widebar{A}_{2n}^{\m(i,j)}\Big)^2\m\bigg] & \sim \frac{1}{8}\bigg(\frac{1}{N}\bigg)\;.
\end{align*}
\end{thm}
The asymptotic convergence rates in Theorem~\ref{thm:asmyp_conv_rates_combined} are phrased in terms of $N$ since the number of Gaussian random vectors required to define the above L\'{e}vy area approximations is $N$ or $N-1$, respectively. Of course, it is straightforward to define the polynomial approximation $\widebar{A}_{n}$ for $n\in\N$, see Theorem \ref{thm:poly_area_error}.\medskip

Intriguingly, the convergence rates for the approximations resulting from the Fourier series and the polynomial expansion correspond exactly with the areas under the limit variance function for each fluctuation process, which are
\begin{equation*}
    \int_0^1\frac{1}{\pi^2}\dd t=\frac{1}{\pi^2}
    \quad\text{and}\quad
    \int_0^1 \frac{1}{\pi}\sqrt{t(1-t)}\dd t=\frac{1}{8}\;.
\end{equation*}
We provide heuristics demonstrating how this correspondence arises at the end of Section~\ref{sect:levyarea}.\medskip

By adding an additional Gaussian random matrix that matches the covariance of the tail sum, it is possible to derive high order L\'{e}vy area approximations with $O(N^{-1})$ convergence in $L^2(\mathbb{P})$.
Wiktorsson \cite{Wiktorsson} proposed this approach using the Kloeden--Platen--Wright approximation~(\ref{eq:kpw_fourier_approx_intro}) and this was recently improved by Mrongowius and R{\"{o}}{\ss}ler in~\cite{RecentLevyArea} who use the approximation~(\ref{eq:levy_fourier_approx_intro}) obtained from the Fourier series expansion~(\ref{eq:bridge_fourier_expansion}).\medskip

We expect that an $O(N^{-1})$ polynomial-based approximation is possible using the same techniques. While this approximation should be slightly less accurate than the Fourier approach, we expect it to be easier to implement due to both the independence of the coefficients $\{c_k\}_{k\in\N}$ and the covariance of the tail sum having a closed-form expression, see Theorem~\ref{thm:poly_area_error}. Moreover, this type of method has already been studied in \cite{Davie, Flint, Fosterthesis} with Brownian L\'{e}vy area being approximated by
\begin{equation}\label{eq:cheap_levy_area}
\wideparen{A}_{\m 0,1}^{\m (i,j)} := \frac{1}{2}\left(W_1^{(i)}c_1^{(j)} - c_1^{(i)}W_1^{(j)}\right) + \lambda_{\m 0,1}^{(i,j)}\;,
\end{equation}
where the antisymmetric $d\times d$ matrix $\lambda_{\m 0,1}$ is normally distributed and designed so that $\wideparen{A}_{\m 0,1}$ has the same covariance structure as the Brownian L\'{e}vy area $A_{\m 0,1}$. Davie \cite{Davie} as well as Flint and Lyons~\cite{Flint} generate each $(i,j)$-entry of $\lambda_{0,1}$ independently as $\lambda_{\m 0,1}^{(i,j)} \sim \mathcal{N}\big(0, \frac{1}{12}\big)$ for $i < j\m$. In \cite{Fosterthesis}, it is shown that the covariance structure of $A_{0,1}$ can be explicitly computed conditional on both $W_1$ and $c_1$.
By matching the conditional covariance structure of $A_{\m 0,1}$, the work~\cite{Fosterthesis} obtains the approximation
\begin{equation*}
\lambda_{\m 0,1}^{(i,j)} \sim \mathcal{N}\bigg(0, \frac{1}{20} + \frac{1}{20}\Big(\big(c_1^{(i)}\big)^2 + \big(c_1^{(j)}\big)^2\Big)\bigg)\;,
\end{equation*}
where the entries $\big\{\lambda_{\m 0,1}^{(i,j)}\big\}_{i\m<\m j}$ are still generated independently, but only after $c_1$ has been generated.\medskip

By rescaling (\ref{eq:cheap_levy_area}) to approximate L\'{e}vy area on $\big[\frac{k}{N}, \frac{k+1}{N}\big]$ and summing over $k\in\{0,\dots, N-1\}$, we obtain a fine discretisation of $A_{0,1}$ involving 
$2N$ Gaussian random vectors and $N$ random matrices.
In \cite{Davie, Flint, Fosterthesis}, the L\'{e}vy area of Brownian motion and this approximation are probabilistically coupled in such a way that $L^{2}(\mathbb{P})$ convergence rates of $O(N^{-1})$ can be established. Furthermore, the efficient L\'{e}vy area approximation (\ref{eq:cheap_levy_area}) can be used directly in numerical methods for SDEs, which then achieve $L^{2}(\mathbb{P})$ convergence of $O(N^{-1})$ under certain conditions on the SDE vector fields, see \cite{Davie, Flint}.
We leave such high order polynomial-based approximations of L\'{e}vy area as a topic for future work.

\medskip

\subsection*{The paper is organised as follows.}
In Section~\ref{sect:expansions}, we provide an overview of the three expansions we consider for the Brownian bridge, and we characterise the associated fluctuation processes $(F_t^{N,1})_{t\in[0,1]}$ and $(F_t^{N,2})_{t\in[0,1]}$. Before discussing their behaviour in the limit $N\to\infty$, we initiate the moment analysis used to prove the on-diagonal part of Theorem~\ref{thm:sinflucutation} and we extend the analysis to determine the values of the Riemann zeta function at even positive integers in Section~\ref{sect:zeta}. The proof of Theorem~\ref{thm:sinflucutation} follows in Section~\ref{sect:fluct}, where we complete the moment analysis and establish a locally uniform convergence to identify the limit on the diagonal, before we deduce Corollary~\ref{cor1}, which then allows us to obtain the off-diagonal convergence in Theorem~\ref{thm:sinflucutation}. We close Section~\ref{sect:fluct} by proving Theorem~\ref{thm:fluctuations}. In Section~\ref{sect:levyarea}, we compare the asymptotic convergence rates of the different approximations of L\'{e}vy area, which results in a proof of Theorem~\ref{thm:asmyp_conv_rates_combined}.

%% file: expansions.tex
We discuss the Karhunen--Lo{\`e}ve expansion as well as the Fourier expansion of the Brownian bridge more closely, and we derive expressions for the covariance functions of their Gaussian fluctuation processes.

\medskip

In our analysis, we frequently use a type of It{\^o} isometry for It{\^o} integrals with respect to a Brownian bridge, and we include its statement and proof for completeness.
\begin{lemma}\label{lem:itoBB}
    Let $(B_t)_{t\in[0,1]}$ be a Brownian bridge in $\R$ with $B_0=B_1=0$, and let $f,g\colon [0,1]\to \R$ be integrable functions. Setting $F(1)=\int_0^1 f(t)\dd t$ and $G(1)=\int_0^1 g(t)\dd t$, we have
    \begin{equation*}
        \E\left[\left(\int_0^1 f(t)\dd B_t\right)\left(\int_0^1 g(t)\dd B_t\right)\right]
        =\int_0^1 f(t)g(t) \dd t - F(1)G(1)\;.
    \end{equation*}
\end{lemma}
\begin{proof}
    For a standard one-dimensional Brownian motion $(W_t)_{t\in[0,1]}$, the process
    $(W_t-t W_1)_{t\in[0,1]}$ has the same law as the Brownian bridge $(B_t)_{t\in[0,1]}$. In particular, the random variable $\int_0^1 f(t)\dd B_t$ is equal in law to the random variable
    \begin{equation*}
        \int_0^1 f(t)\dd W_t -W_1\int_0^1f(t)\dd t
        =\int_0^1 f(t)\dd W_t - W_1 F(1)\;.
    \end{equation*}
    Using a similar expression for $\int_0^1 g(t)\dd B_t$ and applying the usual It\^o isometry, we deduce that
    \begin{align*}
        &\E\left[\left(\int_0^1 f(t)\dd B_t\right)\left(\int_0^1 g(t)\dd B_t\right)\right]\\
        &\qquad=\int_0^1 f(t)g(t) \dd t-F(1)\int_0^1 g(t)\dd t
        -G(1)\int_0^1 f(t)\dd t+F(1)G(1)\\
        &\qquad=\int_0^1 f(t)g(t) \dd t - F(1)G(1)\;,
    \end{align*}
    as claimed.
\end{proof}

\subsection{The Karhunen--Lo{\`e}ve expansion}
Mercer's theorem, see~\cite{mercer}, states that for a continuous symmetric non-negative definite kernel $K\colon[0,1]\times[0,1]\to\R$ there exists an orthonormal basis $\{e_k\}_{k\in\N}$ of $L^2([0,1])$ which consists of eigenfunctions of the Hilbert–Schmidt integral operator associated with $K$ and whose eigenvalues $\{\lambda_k\}_{k\in\N}$ are non-negative and such that, for $s,t\in[0,1]$, we have the representation
\begin{equation*}
    K(s,t)=\sum_{k=1}^\infty \lambda_k e_k(s) e_k(t)\;,
\end{equation*}
which converges absolutely and uniformly on $[0,1]\times[0,1]$.
For the covariance function $K_B$ defined by~(\ref{cov_BB}) of the Brownian bridge $(B_t)_{t\in[0,1]}$, we obtain, for $k\in\N$ and $t\in[0,1]$,
\begin{equation*}
    e_k(t)=\sqrt{2}\sin(k\pi t)
    \quad\text{and}\quad
    \lambda_k=\frac{1}{k^2\pi^2}\;.
\end{equation*}
The Karhunen--Lo{\`e}ve expansion of the Brownian bridge is then given by
\begin{equation*}
    B_t=\sum_{k=1}^\infty \sqrt{2}\sin(k\pi t) Z_k
    \quad\text{where}\quad
    Z_k=\int_0^1 \sqrt{2}\sin(k\pi r) B_r \dd r\;,
\end{equation*}
which after integration by parts yields the expression~(\ref{eq:bridge_kl_expansion}).
Applying Lemma~\ref{lem:itoBB}, we can compute the covariance functions of the associated fluctuation processes $(F_t^{N,1})_{t\in[0,1]}$.
\begin{lemma}\label{lem:cov1}
    The fluctuation process $(F_t^{N,1})_{t\in[0,1]}$ for $N\in\N$ is a zero-mean Gaussian process with
    covariance function $NC_1^N$ where $C_1^N\colon[0,1]\times[0,1]\to\R$ is given by
    \begin{equation*}
        C_1^N(s,t)=\min(s,t)-st-
        \sum_{k=1}^N\frac{2\sin(k\pi s)\sin(k\pi t)}{k^2\pi^2}\;.
    \end{equation*}
\end{lemma}
\begin{proof}
    From the definition~(\ref{KL_fluctuation}), we see that $(F_t^{N,1})_{t\in[0,1]}$ is a zero-mean Gaussian process. Hence, it suffices to determine its covariance function. By Lemma~\ref{lem:itoBB}, we have, for $k,l\in\N$,
    \begin{equation*}
        \E\left[\left(\int_0^1\cos(k\pi r)\dd B_r\right)\left(\int_0^1\cos(l\pi r)\dd B_r\right)\right]
        =\int_0^1\cos(k\pi r)\cos(l\pi r)\dd r=
        \begin{cases}
        \frac{1}{2} & \text{if }k=l\\
        0 & \text{otherwise}
        \end{cases}
    \end{equation*}
    and, for $t\in[0,1]$,
    \begin{equation*}
        \E\left[B_t\int_0^1\cos(k\pi r)\dd B_r\right]
        =\int_0^t\cos(k\pi r)\dd r=\frac{\sin(k\pi t)}{k\pi}\;.
    \end{equation*}
    Therefore, from~(\ref{cov_BB}) and (\ref{KL_fluctuation}), we obtain that, for all $s,t\in[0,1]$,
    \begin{equation*}
        \E\left[F_s^{N,1}F_t^{N,1}\right]
        =N\left(\min(s,t)-st -\sum_{k=1}^N\frac{2\sin(k\pi s)\sin(k\pi t)}{k^2\pi^2}\right)\;,
    \end{equation*}
    as claimed.
\end{proof}
Consequently, Theorem~\ref{thm:sinflucutation} is a statement about the pointwise convergence of the function $NC_1^N$ in the limit $N\to\infty$.

\medskip

For our stand-alone derivation of the values of the Riemann zeta function at even positive integers in Section~\ref{sect:zeta}, it is further important to note that since, by Mercer's theorem, the representation
\begin{equation}\label{eq:mercer4bridge}
    K_B(s,t)=\min(s,t)-st=\sum_{k=1}^\infty \frac{2\sin(k\pi s)\sin(k\pi t)}{k^2\pi^2}
\end{equation}
converges uniformly for $s,t\in[0,1]$, the sequence $\{C_1^N\}_{N\in\N}$ converges uniformly on $[0,1]\times[0,1]$ to the zero function. It follows that, for all $n\in\N_0$,
\begin{equation}\label{eq:conv_of_moments_4_C1}
    \lim_{N\to\infty}\int_0^1 C_1^N(t,t) t^n\dd t = 0\;.
\end{equation}

\subsection{The Fourier expansion}
Whereas for the Karhunen--Lo{\`e}ve expansion the sequence
\begin{equation*}
    \left\{\int_0^1\cos(k\pi r)\dd B_r\right\}_{k\in\N}
\end{equation*}
of random coefficients is formed by independent Gaussian random variables, it is crucial to observe that the random coefficients appearing in the Fourier expansion are not independent.
Integrating by parts, we can rewrite the coefficients defined in~(\ref{eq:fourier_coefficients}) as
\begin{equation}\label{eq:a0}
    a_0=2\int_0^1B_r\dd r=-2\int_0^1r\dd B_r
    \quad\text{and}\quad b_0=0
\end{equation}
as well as, for $k\in\N$,
\begin{equation}\label{eq:akbk}
    a_k=-\int_0^1\frac{\sin(2k\pi r)}{k\pi}\dd B_r
    \quad\text{and}\quad
    b_k=\int_0^1\frac{\cos(2k\pi r)}{k\pi}\dd B_r\;.
\end{equation}
Applying Lemma~\ref{lem:itoBB}, we see that
\begin{equation}\label{eq:a02}
    \E\left[a_0^2\right]=4\left(\int_0^1 r^2\dd r-\frac{1}{4}\right)=\frac{1}{3}
\end{equation}
and, for $k,l\in\N$,
\begin{equation}\label{eq:akal}
    \E\left[a_k a_l\right]=\E\left[b_k b_l\right]=
    \begin{cases}
    \dfrac{1}{2k^2\pi^2} & \text{if }k=l\\[6pt]
    0 & \text{otherwise}
    \end{cases}\;.
\end{equation}
Since the random coefficients are Gaussian random variables with mean zero, by~(\ref{eq:a0}) and (\ref{eq:akbk}), this implies that, for $k\in\N$,
\begin{equation*}
    a_0\sim\mathcal{N}\left(0,\frac{1}{3}\right)
    \quad\text{and}\quad
    a_k,b_k\sim\mathcal{N}\left(0,\frac{1}{2k^2\pi^2}\right)\;.
\end{equation*}
For the remaining covariances of these random coefficients, we obtain that, for $k,l\in\N$,
\begin{equation}\label{eq:remaining_covariances}
    \E\left[a_k b_l\right]=0\;,\quad
    \E\left[a_0a_k\right]=2\int_0^1\frac{\sin(2k\pi r)}{k\pi}r\dd r=-\frac{1}{k^2\pi^2}
    \quad\text{and}\quad
    \E\left[a_0b_k\right]=0\;.
\end{equation}
Using the covariance structure of the random coefficients, we determine the covariance functions of the fluctuation processes $(F_t^{N,2})_{t\in[0,1]}$ defined in~(\ref{fourier_fluctuation}) for the Fourier series expansion.
\begin{lemma}\label{lem:cov2}
    The fluctuation process $(F_t^{N,2})_{t\in[0,1]}$ for $N\in\N$ is a Gaussian process with mean zero and whose covariance function is $2NC_2^N$ where $C_2^N\colon[0,1]\times[0,1]$ is given by
    \begin{equation*}
        C_2^N(s,t)=\min(s,t)-st+\frac{s^2-s}{2}+\frac{t^2-t}{2}+\frac{1}{12}- \sum_{k=1}^N\frac{\cos(2k\pi(t-s))}{2k^2\pi^2}\;.
\end{equation*}
\end{lemma}
\begin{proof}
    Repeatedly applying Lemma~\ref{lem:itoBB}, we compute that, for $t\in[0,1]$,
    \begin{equation}\label{eq:Bta0}
        \E\left[B_ta_0\right]=-2\int_0^t r\dd r+\int_0^t\dd r=t-t^2
    \end{equation}
    as well as, for $k\in\N$,
    \begin{equation}\label{eq:BtakBtbk}
        \E\left[B_ta_k\right]=-\int_0^t\frac{\sin(2k\pi r)}{k\pi}\dd r
        =\frac{\cos(2k\pi t)-1}{2k^2\pi^2}
        \quad\text{and}\quad
        \E\left[B_tb_k\right]=\frac{\sin(2k\pi t)}{2k^2\pi^2}\;.
    \end{equation}
    From~(\ref{eq:a02}) and (\ref{eq:Bta0}), it follows that, for $s,t\in[0,1]$,
    \begin{equation*}
        \E\left[\left(B_s-\frac{1}{2}a_0\right)\left(B_t-\frac{1}{2}a_0\right)\right]
        =\min(s,t)-st+\frac{s^2-s}{2}+\frac{t^2-t}{2}+\frac{1}{12}\;,
    \end{equation*}
    whereas~(\ref{eq:remaining_covariances}) and (\ref{eq:BtakBtbk}) imply that
    \begin{equation*}
        \E\left[\frac{1}{2}a_0\sum_{k=1}^N a_k \cos(2k\pi t) -B_s\sum_{k=1}^N a_k \cos(2k\pi t)\right]
        =-\sum_{k=1}^N\frac{\cos(2k\pi s)\cos(2k\pi t)}{2k^2\pi^2}
    \end{equation*}
    as well as
    \begin{equation*}
        \E\left[B_s\sum_{k=1}^N b_k \sin(2k\pi t)\right]
        =\sum_{k=1}^N\frac{\sin(2k\pi s)\sin(2k\pi t)}{2k^2\pi^2}\;.
    \end{equation*}
    It remains to observe that, by~(\ref{eq:akal}) and (\ref{eq:remaining_covariances}),
    \begin{align*}
        &\E\left[\left(\sum_{k=1}^N\left(a_k\cos(2k\pi s)+b_k\sin(2k\pi s)\right)\right) \left(\sum_{k=1}^N\left(a_k\cos(2k\pi t)+b_k\sin(2k\pi t)\right)\right)\right]\\
        &\qquad=
        \sum_{k=1}^N\frac{\cos(2k\pi s)\cos(2k\pi t)+\sin(2k\pi s)\sin(2k\pi t)}{2k^2\pi^2}\;.
    \end{align*}
    Using the identity
    \begin{equation}\label{eq:trig_identity}
        \cos(2k\pi (t-s))=\cos(2k\pi s)\cos(2k\pi t)+\sin(2k\pi s)\sin(2k\pi t)
    \end{equation}
    and recalling the definition~(\ref{fourier_fluctuation}) of the fluctuation process $(F_t^{N,2})_{t\in[0,1]}$ for the Fourier expansion, we obtain the desired result.
\end{proof}
By combining Corollary~\ref{cor1}, the resolution~(\ref{eq:basel}) to the Basel problem and the representation~(\ref{eq:mercer4bridge}), we can determine the pointwise limit of $2N C_2^N$ as $N\to\infty$.
We leave further considerations until Section~\ref{sect:fluct2} to demonstrate that the identity~(\ref{eq:basel}) is really a consequence of our analysis.

\subsection{The polynomial expansion}
\label{sect:poly}
As pointed out in the introduction and as discussed in detail in~\cite{foster}, the polynomial expansion of the Brownian bridge is a type of Karhunen--Lo{\`e}ve expansion in the weighted $L^2(\pr)$ space with weight function $w$ on $(0,1)$ defined by $w(t)=\frac{1}{t(1-t)}$.\medskip

An alternative derivation of the polynomial expansion is given in~\cite{semicircle} by considering iterated Kolmogorov diffusions. The iterated Kolmogorov diffusion of step $N\in\N$ pairs a one-dimensional Brownian motion $(W_t)_{t\in[0,1]}$ with its first $N-1$ iterated time integrals, that is, it is the stochastic process in $\R^N$ of the form
\begin{equation*}
    \left(W_t,\int_0^t W_{s_1}\dd s_1,\dots,
    \int_0^t\int_0^{s_{N-1}}\dots \int_0^{s_2} W_{s_1}\dd s_1\dots\dd s_{N-1}\right)_{t\in[0,1]}\;.
\end{equation*}
The shifted Legendre polynomial $Q_k$ of degree $k\in\N$ on the interval $[0,1]$ is defined in terms of the standard Legendre polynomial $P_k$ of degree $k$ on $[-1,1]$ by, for $t\in[0,1]$,
\begin{equation*}
    Q_k(t)=P_k(2t-1)\;.
\end{equation*}
It is then shown that the first component of an iterated Kolmogorov diffusion of step $N\in\N$ conditioned to return to $0\in\R^N$ in time $1$ has the same law as the stochastic process
\begin{equation*}
    \left(B_t-\sum_{k=1}^{N-1}(2k+1)\int_0^tQ_k(r)\dd r \int_0^1Q_k(r)\dd B_r\right)_{t\in[0,1]}\;.
\end{equation*}
The polynomial expansion~(\ref{eq:bridge_poly_expansion}) is an immediate consequence of the result~\cite[Theorem~1.4]{semicircle} which states
that these first components of the conditioned iterated Kolmogorov diffusions converge weakly as $N\to\infty$ to the zero process.\medskip

As for the Karhunen--Lo{\`e}ve expansion discussed above, the sequence $\{c_k\}_{k\in\N}$ of random coefficients defined by~(\ref{eq:poly_coefficients}) is again formed by independent Gaussian random variables. To see this, we first recall the following identities for Legendre polynomials~\cite[(12.23), (12.31), (12.32)]{ArfkenWeber} which in terms of the shifted Legendre polynomials read as, for $k\in\N$,
\begin{equation}\label{eq:id_shiftedLP}
Q_k = \frac{1}{2(2k+1)}\left(Q_{k+1}^{\prime} - Q_{k-1}^{\prime}\right)\;,\qquad
Q_k(0) = (-1)^k\;,\qquad
Q_k(1) = 1\;.
\end{equation}
In particular, it follows that, for all $k\in\N$,
\begin{equation*}
    \int_0^1 Q_k(r)\dd r=0\;,
\end{equation*}
which, by Lemma~\ref{lem:itoBB}, implies that, for $k,l\in\N$,
\begin{equation*}
    \E\left[c_k c_l\right]
    =\E\left[\left(\int_0^1 Q_k(r)\dd B_r\right)\left(\int_0^1 Q_l(r)\dd B_r\right)\right]
    =\int_0^1 Q_k(r) Q_l(r)\dd r=
    \begin{cases}
    \dfrac{1}{2k+1} & \text{if } k=l\\[6pt]
    0 & \text{otherwise}
    \end{cases}\;.
\end{equation*}
Since the random coefficients are Gaussian with mean zero, this establishes their independence.

\medskip

The fluctuation processes $(F_t^{N,3})_{t\in[0,1]}$ for the polynomial expansion defined by
\begin{equation}\label{poly_fluctuation}
    F_t^{N,3}=\sqrt{N}\left(B_t-\sum_{k=1}^{N-1}(2k+1) \int_0^t Q_k(r)\dd r \int_0^1 Q_k(r)\dd B_r\right)
\end{equation}
are studied in~\cite{semicircle}. According to~\cite[Theorem 1.6]{semicircle}, they converge in finite dimensional distributions as $N\to\infty$ to the collection $(F_t^3)_{t\in[0,1]}$ of independent Gaussian random variables with mean zero and variance
\begin{equation*}
\E\left[\left(F_t^3\right)^2\right]=\frac{1}{\pi}\sqrt{t(1-t)}\;,
\end{equation*}
that is, the variance function of the limit fluctuations is given by a scaled semicircle.

%% file: zeta.tex
We demonstrate how to use the Karhunen--Lo{\`e}ve expansion of the Brownian bridge or, more precisely, the series representation arising from Mercer's theorem for the covariance function of the Brownian bridge to determine the values of the Riemann zeta function at even positive integers. The analysis further feeds directly into Section~\ref{sect:fluct1} where we characterise the limit fluctuations for the Karhunen--Lo{\`e}ve expansion.

\medskip

The crucial ingredient is the observation~(\ref{eq:conv_of_moments_4_C1}) from Section~\ref{sect:expansions}, which implies that, for all $n\in\N_0$,
\begin{equation}\label{eq:moment_relation}
    \sum_{k=1}^\infty\int_0^1 \frac{2\left(\sin(k\pi t)\right)^2}{k^2\pi^2}t^{n} \dd t
    =\int_0^1\left(t-t^2\right)t^{n}\dd t
    =\frac{1}{(n+2)(n+3)}\;.
\end{equation}
For completeness, we recall that the Riemann zeta function $\zeta\colon\C\setminus\{1\}\to\C$
analytically continues the sum of the Dirichlet series
\begin{equation*}
  \zeta(s)=\sum_{k=1}^\infty\frac{1}{k^s}\;.
\end{equation*}
When discussing its values at even positive integers, we encounter the Bernoulli numbers. The Bernoulli numbers $B_n$, for $n\in\N$, are signed rational numbers defined by an exponential generating function via, for $t\in(-2\pi,2\pi)$,
\begin{equation*}
    \frac{t}{\e^t-1}=1+\sum_{n=1}^\infty\frac{B_n t^n}{n!}\;,
\end{equation*}
see Borevich and Shafarevich~\cite[Chapter~5.8]{zeta_bernoulli}. These numbers play an important role in number theory and analysis. For instance, they feature in the series expansion of the (hyperbolic) tangent and the (hyperbolic) cotangent, and they appear in formulae by Bernoulli and by Faulhaber for the sum of positive integer powers of the first $k$ positive integers. The characterisation of the Bernoulli numbers which is essential to our analysis is that, according to~\cite[Theorem~5.8.1]{zeta_bernoulli}, they satisfy and are uniquely given by the recurrence relations
\begin{equation}\label{eq:recurrence_bernoulli}
    1+\sum_{n=1}^{m}\binom{m+1}{n}B_n=0\quad\text{for }m\in\N\;.
\end{equation}
In particular, choosing $m=1$ yields $1+2B_1=0$, which shows that
\begin{equation*}
    B_{1}=-\frac{1}{2}\;.
\end{equation*}
Moreover, since the function defined by, for $t\in(-2\pi,2\pi)$,
\begin{equation*}
    \frac{t}{\e^t-1}+\frac{t}{2}=1+\sum_{n=2}^\infty\frac{B_n t^n}{n!}
\end{equation*}
is an even function, we obtain $B_{2n+1}=0$ for all $n\in\N$, see~\cite[Theorem~5.8.2]{zeta_bernoulli}.
It follows from~(\ref{eq:recurrence_bernoulli}) that the Bernoulli numbers $B_{2n}$ indexed by even positive integers are uniquely characterised by the recurrence relations
\begin{equation}\label{eq:char_bernoulli_2n}
    \sum_{n=1}^m \binom{2m+1}{2n}B_{2n}=\frac{2m-1}{2}\quad\text{for }m\in\N\;.
\end{equation}
These recurrence relations are our tool for identifying the Bernoulli numbers when determining the values of the Riemann zeta function at even positive integers.

\medskip

The starting point for our analysis is~(\ref{eq:moment_relation}), and we first illustrate how it allows us to compute $\zeta(2)$. Taking $n=0$ in~(\ref{eq:moment_relation}), multiplying through by $\pi^2$, and using that $\int_0^1\left(\sin(k\pi t)\right)^2\dd t=\frac{1}{2}$ for $k\in\N$, we deduce that
\begin{equation*}
    \zeta(2)=\sum_{k=1}^\infty\frac{1}{k^2}
    =\sum_{k=1}^\infty\int_0^1 \frac{2\left(\sin(k\pi t)\right)^2}{k^2} \dd t
    =\frac{\pi^2}{6}\;.
\end{equation*}
We observe that this is exactly the identity obtained by applying the general result
\begin{equation*}
    \int_0^1 K(t,t)\dd t=\sum_{k=1}^\infty \lambda_k
\end{equation*}
for a representation arising from Mercer's theorem to the representation for the covariance function $K_B$ of the Brownian bridge.

\medskip

For working out the values for the remaining even positive integers, we iterate over the degree of the moment in~(\ref{eq:moment_relation}). While for the remainder of this section it suffices to only consider the even moments, we derive the following recurrence relation and the explicit expression both for the even and for the odd moments as these are needed in Section~\ref{sect:fluct1}.
For $k\in\N$ and $n\in\N_0$, we set
\begin{equation*}
    e_{k,n}=\int_0^1 2\left(\sin(k\pi t)\right)^2t^{n}\dd t\;.
\end{equation*}
\begin{lemma}\label{lem:moment_recurrence}
    For all $k\in\N$ and all $n\in\N$ with $n\geq 2$, we have
    \begin{equation*}
        e_{k,n}=\frac{1}{n+1}-\frac{n(n-1)}{4k^2\pi^2}e_{k,n-2}
    \end{equation*}
    subject to the initial conditions
    \begin{equation*}
        e_{k,0}=1\quad\text{and}\quad e_{k,1}=\frac{1}{2}\;.
    \end{equation*}
\end{lemma}
\begin{proof}
    For $k\in\N$, the values for $e_{k,0}$ and $e_{k,1}$ can be verified directly. For $n\in\N$ with $n\geq 2$, we integrate by parts twice to obtain
    \begin{align*}
        e_{k,n} &=\int_0^1 2\left(\sin(k\pi t)\right)^2t^{n}\dd t\\
        &=1-\int_0^1\left(t-\frac{\sin(2k\pi t)}{2k\pi}\right) nt^{n-1}\dd t\\
        &=1-\frac{n}{2}+\frac{n(n-1)}{2} \int_0^1\left(t^2-
          \frac{\left(\sin(k\pi t)\right)^2}{k^2\pi^2}\right)t^{n-2}\dd t \\
        &=\frac{2-n}{2}+\frac{n(n-1)}{2}
          \left(\frac{1}{n+1}-\frac{1}{2k^2\pi^2}e_{k,n-2}\right)\\
        &=\frac{1}{n+1}-\frac{n(n-1)}{4k^2\pi^2}e_{k,n-2}\;,
  \end{align*}
  as claimed.
\end{proof}
Iteratively applying the recurrence relation, we find the following explicit expression, which despite its involvedness is exactly what we need.
\begin{lemma}\label{lem:moment_expression}
  For all $k\in\N$ and $m\in\N_0$, we have
  \begin{align*}
    e_{k,2m}&=\frac{1}{2m+1}+
    \sum_{n=1}^m\frac{(-1)^n(2m)!}{(2(m-n)+1)!2^{2n}}\frac{1}{k^{2n}\pi^{2n}}
    \quad\text{and}\\
    e_{k,2m+1}&=\frac{1}{2m+2}+
    \sum_{n=1}^m\frac{(-1)^n(2m+1)!}{(2(m-n)+2)!2^{2n}}\frac{1}{k^{2n}\pi^{2n}}\;.
  \end{align*}
\end{lemma}
\begin{proof}
  We proceed by induction over $m$. Since $e_{k,0}=1$ and $e_{k,1}=\frac{1}{2}$ for all $k\in\N$, the expressions are true for $m=0$ with the sums being understood as empty sums in this case. Assuming that the result is true for some fixed $m\in\N_0$, we use Lemma~\ref{lem:moment_recurrence} to deduce that
  \begin{align*}
    e_{k,2m+2}&=\frac{1}{2m+3}-\frac{(2m+2)(2m+1)}{4k^2\pi^2}e_{k,2m}\\
    &=\frac{1}{2m+3}-\frac{2m+2}{4k^2\pi^2}-
      \sum_{n=1}^m\frac{(-1)^n(2m+2)!}{(2(m-n)+1)!2^{2n+2}}\frac{1}{k^{2n+2}\pi^{2n+2}}\\
    &=\frac{1}{2m+3}+
      \sum_{n=1}^{m+1}\frac{(-1)^n(2m+2)!}{(2(m-n)+3)!2^{2n}}\frac{1}{k^{2n}\pi^{2n}}    
  \end{align*}
  as well as
  \begin{align*}
    e_{k,2m+3}&=\frac{1}{2m+4}-\frac{(2m+3)(2m+2)}{4k^2\pi^2}e_{k,2m+1}\\
    &=\frac{1}{2m+4}-\frac{2m+3}{4k^2\pi^2}-
      \sum_{n=1}^m\frac{(-1)^n(2m+3)!}{(2(m-n)+2)!2^{2n+2}}\frac{1}{k^{2n+2}\pi^{2n+2}}\\
    &=\frac{1}{2m+4}+
      \sum_{n=1}^{m+1}\frac{(-1)^n(2m+3)!}{(2(m-n)+4)!2^{2n}}\frac{1}{k^{2n}\pi^{2n}}\;,
  \end{align*}
  which settles the induction step.
\end{proof}
Focusing on the even moments for the remainder of this section, we see that by~(\ref{eq:moment_relation}), for all $m\in\N_0$,
\begin{equation*}
    \sum_{k=1}^\infty\frac{e_{k,2m}}{k^2\pi^2}
    =\frac{1}{(2m+2)(2m+3)}\;.
\end{equation*}
From Lemma~\ref{lem:moment_expression}, it follows that
\begin{equation*}
    \sum_{k=1}^\infty\frac{1}{k^2\pi^2}
    \left(\sum_{n=0}^m\frac{(-1)^n(2m)!}{(2(m-n)+1)!2^{2n}}\frac{1}{k^{2n}\pi^{2n}}\right)
    =\frac{1}{(2m+2)(2m+3)}\;.
\end{equation*}
Since $\sum_{k=1}^\infty k^{-2n}$ converges for all $n\in\N$, we can rearrange sums to obtain
\begin{equation*}
    \sum_{n=0}^m\frac{(-1)^n(2m)!}{(2(m-n)+1)!2^{2n}}
    \left(\sum_{k=1}^\infty\frac{1}{k^{2n+2}\pi^{2n+2}}\right)
    =\frac{1}{(2m+2)(2m+3)}\;,
\end{equation*}
which in terms of the Riemann zeta function and after reindexing the sum rewrites as
\begin{equation*}
    \sum_{n=1}^{m+1}\frac{(-1)^{n+1}(2m)!}{(2(m-n)+3)!2^{2n-2}}
    \frac{\zeta(2n)}{\pi^{2n}}
    =\frac{1}{(2m+2)(2m+3)}\;.
\end{equation*}
Multiplying through by $(2m+1)(2m+2)(2m+3)$ shows that, for all $m\in\N_0$,
\begin{equation*}
  \sum_{n=1}^{m+1}\binom{2m+3}{2n}
  \left(\frac{(-1)^{n+1}2(2n)!}{\left(2\pi\right)^{2n}}\zeta(2n)\right)
  =\frac{2m+1}{2}\;.
\end{equation*}
Comparing the last expression with the characterisation~(\ref{eq:char_bernoulli_2n}) of the Bernoulli numbers $B_{2n}$ indexed by even positive integers implies that
\begin{equation*}
    B_{2n}=\frac{(-1)^{n+1}2(2n)!}{\left(2\pi\right)^{2n}}\zeta(2n)\;,
\end{equation*}
that is, we have established that, for all $n\in\N$,
\begin{equation*}
  \zeta(2n)=(-1)^{n+1}\frac{\left(2\pi\right)^{2n}B_{2n}}{2(2n)!}\;.
\end{equation*}

%% file: fluctuations1.tex
For the moment analysis initiated in the previous section to allow us to identify the limit of $NC_1^N$ as $N\to\infty$ on the diagonal away from its endpoints, we apply the Arzel{\`a}--Ascoli theorem to guarantee continuity of the limit away from the endpoints. To this end, we first need to establish the uniform boundedness of two families of functions. Recall that the functions $C_1^N\colon[0,1]\times[0,1]\to\R$ are defined in Lemma~\ref{lem:cov1}.
\begin{lemma}\label{lem:uniform_bound1}
  The family $\{NC_1^N(t,t)\colon N\in\N\text{ and }t\in[0,1]\}$
  is uniformly bounded.
\end{lemma}
\begin{proof}
    Combining the expression for $C_1^N(t,t)$ from Lemma~\ref{lem:cov1} and the representation~(\ref{eq:mercer4bridge}) for $K_B$ arising from Mercer's theorem, we see that
    \begin{equation*}
        NC_1^N(t,t)=N\sum_{k=N+1}^\infty\frac{2\left(\sin(k\pi t)\right)^2}{k^2\pi^2}\;.
    \end{equation*}
    In particular, for all $N\in\N$ and all $t\in[0,1]$, we have
    \begin{equation*}
        \left|NC_1^N(t,t)\right|\leq N\sum_{k=N+1}^\infty\frac{2}{k^2\pi^2}\;.
    \end{equation*}
    We further observe that
    \begin{equation}\label{eq:idseries_bound}
        \lim_{M\to\infty}N\sum_{k=N+1}^M\frac{1}{k^2}\leq \lim_{M\to\infty} N\sum_{k=N+1}^M\left(\frac{1}{k-1}-\frac{1}{k}\right)
        =\lim_{M\to\infty}\left(1-\frac{N}{M}\right)=1\;.
    \end{equation}
    It follows that, for all $N\in\N$ and all $t\in[0,1]$,
    \begin{equation*}
        \left|NC_1^N(t,t)\right|\leq\frac{2}{\pi^2}\;,
    \end{equation*}
    which is illustrated in Figure~\ref{diag:boundedness1} and which establishes the claimed uniform boundedness.
\end{proof}
\begin{figure}[h]
    \centering
    \includegraphics[width=\textwidth]{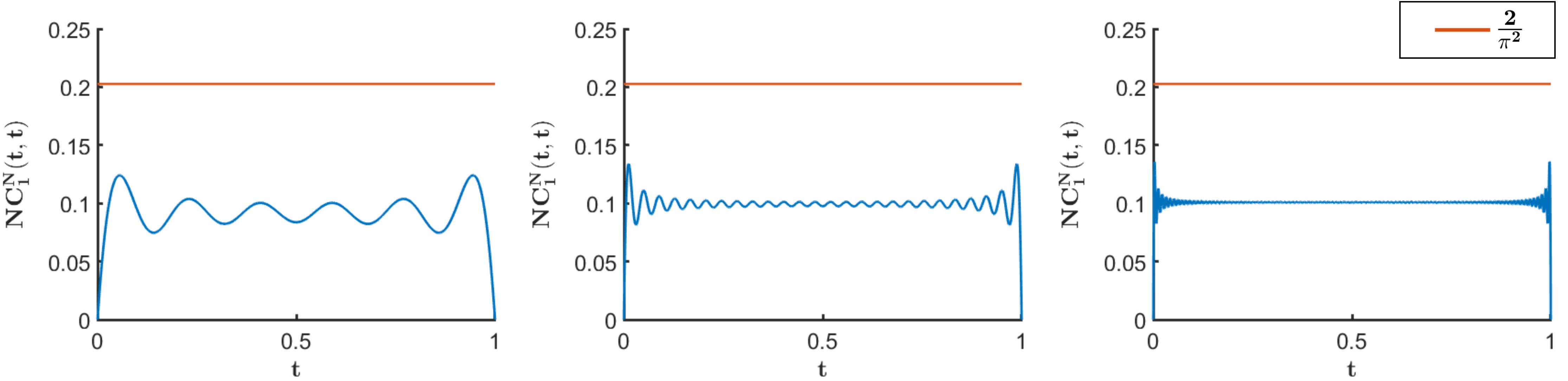}
    \caption{Profiles of $t\mapsto NC_1^N(t,t)$ plotted for $N\in\{5, 25, 100\}$
             along with $t\mapsto \frac{2}{\pi^2}$\;.}
    \label{diag:boundedness1}
\end{figure}

\begin{lemma}\label{lem:uniform_bound2}
  Fix $\eps>0$. The family
  \begin{equation*}
    \left\{N\frac{\db}{\db t}C_1^N(t,t)\colon
      N\in\N\text{ and }t\in[\eps,1-\eps]\right\}
  \end{equation*}
  is uniformly bounded.
\end{lemma}
\begin{proof}
    According to Lemma~\ref{lem:cov1}, we have, for all $t\in[0,1]$,
    \begin{equation*}
        C_1^N(t,t)=t-t^2-\sum_{k=1}^N \frac{2\left(\sin(k\pi t)\right)^2}{k^2\pi^2}\;,
    \end{equation*}
    which implies that
    \begin{equation*}
        N\frac{\db}{\db t}C_1^N(t,t)=N\left(1-2t-\sum_{k=1}^N\frac{2\sin(2k\pi t)}{k\pi}\right)\;.
    \end{equation*}
    The desired result then follows by showing that, for $\eps>0$ fixed, the family
    \begin{equation*}
        \left\{N\left(\frac{\pi-t}{2}-\sum_{k=1}^N\frac{\sin(kt)}{k}\right)\colon
        N\in\N\text{ and }t\in[\eps,2\pi-\eps]\right\}
    \end{equation*}
    is uniformly bounded, as illustrated in Figure~\ref{diag:boundedness2}.
    Employing a usual approach, we use the Dirichlet kernel, for $N\in\N$,
    \begin{equation*}
        \sum_{k=-N}^N\e^{\im kt}=1+\sum_{k=1}^N2\cos(kt)=    \frac{\sin\left(\left(N+\frac{1}{2}\right)t\right)}{\sin\left(\frac{t}{2}\right)}
    \end{equation*}
    to write, for $t\in(0,2\pi)$,
    \begin{equation*}
        \frac{\pi-t}{2}-\sum_{k=1}^N\frac{\sin(kt)}{k}    =-\frac{1}{2}\int_\pi^t\left(1+\sum_{k=1}^N2\cos(ks)\right)\db s
        =-\frac{1}{2}\int_\pi^t\frac{\sin\left(\left(N+\frac{1}{2}\right)s\right)}
        {\sin\left(\frac{s}{2}\right)}\dd s\;.
    \end{equation*}
    Integration by parts yields
    \begin{equation*}
        -\frac{1}{2}\int_\pi^t\frac{\sin\left(\left(N+\frac{1}{2}\right)s\right)}    {\sin\left(\frac{s}{2}\right)}\dd s
        =\frac{\cos\left(\left(N+\frac{1}{2}\right)t\right)}{(2N+1)\sin\left(\frac{t}{2}\right)}
         -\frac{1}{2N+1}\int_\pi^t\cos\left(\left(N+\frac{1}{2}\right)s\right)\frac{\db}{\db s}
         \left(\frac{1}{\sin\left(\frac{s}{2}\right)}\right)\dd s\;.
    \end{equation*}
    By the first mean value theorem for definite integrals, it follows that for $t\in(0,\pi]$ fixed, there exists $\xi\in[t,\pi]$, whereas for $t\in[\pi,2\pi)$ fixed, there exists $\xi\in[\pi,t]$, such that
    \begin{equation*}
        -\frac{1}{2}\int_\pi^t\frac{\sin\left(\left(N+\frac{1}{2}\right)s\right)}    {\sin\left(\frac{s}{2}\right)}\dd s
        =\frac{\cos\left(\left(N+\frac{1}{2}\right)t\right)}{(2N+1)\sin\left(\frac{t}{2}\right)}
         -\frac{\cos\left(\left(N+\frac{1}{2}\right)\xi\right)}{2N+1}
          \left(\frac{1}{\sin\left(\frac{t}{2}\right)}-1\right)\;.
    \end{equation*}
    Since $\left|\cos\left(\left(N+\frac{1}{2}\right)\xi\right)\right|$ is bounded above by one independently of $\xi$ and as $\frac{t}{2}\in(0,\pi)$ for $t\in(0,2\pi)$ implies that $0<\sin\left(\frac{t}{2}\right)\leq 1$, we conclude that, for all $N\in\N$ and for all $t\in(0,2\pi)$,
    \begin{equation*}
        N\left|\frac{\pi-t}{2}-\sum_{k=1}^N\frac{\sin(kt)}{k}\right|
        \leq \frac{2N}{(2N+1)\sin\left(\frac{t}{2}\right)}\;,
    \end{equation*}
    which, for $t\in[\eps,2\pi-\eps]$, is uniformly bounded by $1/\sin\left(\frac{\eps}{2}\right)$\;.
\end{proof}
\begin{figure}[h]
    \centering
    \includegraphics[width=0.8\textwidth]{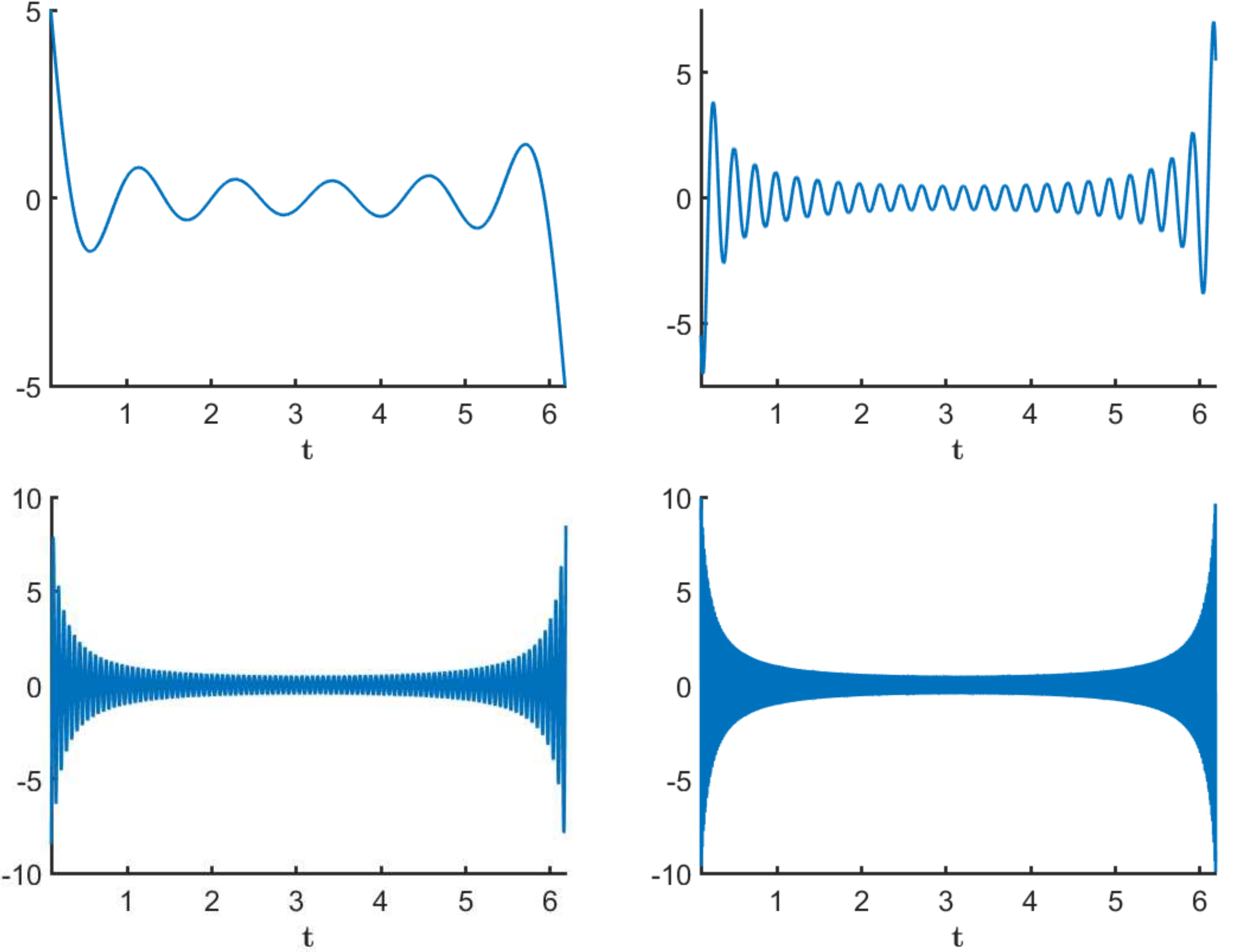}
    \caption{Profiles of $t\mapsto N\bigg(\frac{\pi-t}{2}-\sum\limits_{k=1}^N
    \frac{\sin(kt)}{k}\bigg)$ plotted for $N\in\{5, 25, 100, 1000\}$ on $[\eps, 2\pi - \eps]$ with $\eps = 0.1$\,.}
    \label{diag:boundedness2}
\end{figure} 
\begin{remark0}\rm
    In the proof of the previous lemma, we have essentially controlled the error in the Fourier series expansion for the fractional part of $t$ which is given by
    \begin{equation*}
        \frac{1}{2}-\sum_{k=1}^\infty\frac{\sin(2k\pi t)}{k\pi}\;,
    \end{equation*}
    see~\cite[Exercise on p.~4]{iwaniec}.
\end{remark0}
We can now prove the convergence in Theorem~\ref{thm:sinflucutation} on the diagonal away from the endpoints, which consists of a moment analysis to identify the moments of the limit function as well as an application of the Arzel{\`a}--Ascoli theorem to show that the limit function is continuous away from the endpoints. Alternatively, one could prove Corollary~\ref{cor1} directly with a similar approach as in the proof of Lemma~\ref{lem:uniform_bound2}, but integrating the Dirichlet kernel twice, and then deduce Theorem~\ref{thm:sinflucutation}. However, as the moment analysis was already set up in Section~\ref{sect:zeta} to determine the values of the Riemann zeta function at even positive integers, we demonstrate how to proceed with this approach.
\begin{propn}\label{propn:ondiagonal}
    For all $t\in(0,1)$, we have
    \begin{equation*}
        \lim_{N\to\infty}
        N\left(t-t^2-\sum_{k=1}^N \frac{2\left(\sin(k\pi t)\right)^2}{k^2\pi^2}\right)
        =\frac{1}{\pi^2}\;.
    \end{equation*}
\end{propn}
\begin{proof}
    Recall that, due Lemma~\ref{lem:cov1} and the representation~(\ref{eq:mercer4bridge}), we have, for $t\in[0,1]$,
    \begin{equation}\label{eq:seriesrepforC1N}
        C_1^N(t,t)=t-t^2-\sum_{k=1}^N \frac{2\left(\sin(k\pi t)\right)^2}{k^2\pi^2}
        =\sum_{k=N+1}^\infty \frac{2\left(\sin(k\pi t)\right)^2}{k^2\pi^2}\;.
    \end{equation}
    By Lemma~\ref{lem:uniform_bound1} and Lemma~\ref{lem:uniform_bound2},
    the Arzel{\`a}--Ascoli theorem can be applied locally to any subsequence of $\{N C_1^N\}_{N\in\N}$. Repeatedly using the Arzel\`{a}--Ascoli theorem and a diagonal argument, we deduce that there exists a subsequence of $\{N C_1^N\}_{N\in\N}$ which converges pointwise to a continuous limit function on the interval $(0,1)$. To prove that the full sequence converges pointwise and to identify the limit function, we proceed with the moment analysis initiated in Section~\ref{sect:zeta}. Applying Lemma~\ref{lem:moment_expression}, we see that, for $m\in\N_0$,
    \begin{align}
        N\sum_{k=N+1}^\infty\frac{e_{k,2m}}{k^2\pi^2}\label{eq:evenmoments}
        &=N\sum_{k=N+1}^\infty\frac{1}{k^2\pi^2}
          \left(\frac{1}{2m+1}+
          \sum_{n=1}^m\frac{(-1)^n(2m)!}{(2(m-n)+1)!2^{2n}}\frac{1}{k^{2n}\pi^{2n}}\right)\;,\\
        N\sum_{k=N+1}^\infty\frac{e_{k,2m+1}}{k^2\pi^2}\label{eq:oddmoments}
        &=N\sum_{k=N+1}^\infty\frac{1}{k^2\pi^2}
          \left(\frac{1}{2m+2}+    \sum_{n=1}^m\frac{(-1)^n(2m+1)!}{(2(m-n)+2)!2^{2n}}\frac{1}{k^{2n}\pi^{2n}}\right)\;.
    \end{align}
    The bound~(\ref{eq:idseries_bound}) together with
    \begin{equation*}
        \lim_{M\to\infty}N\sum_{k=N+1}^M\frac{1}{k^2}
        \geq \lim_{M\to\infty}N\sum_{k=N+1}^M\left(\frac{1}{k}-\frac{1}{k+1}\right)  =\lim_{M\to\infty}\left(\frac{N}{N+1}-\frac{N}{M+1}\right)
        =\frac{N}{N+1}
    \end{equation*}
    implies that
    \begin{equation}\label{eq:Nk21}
        \lim_{N\to\infty}N\sum_{k=N+1}^\infty\frac{1}{k^2}=1\;.
    \end{equation}
    For $n\in\N$, we further have
    \begin{equation*}
        0\leq N\sum_{k=N+1}^\infty\frac{1}{k^{2n+2}}
        \leq \frac{N}{(N+1)^2}\sum_{k=N+1}^\infty\frac{1}{k^{2n}}
        \leq \frac{1}{N}\sum_{k=1}^\infty\frac{1}{k^{2n}}\;,
    \end{equation*}
    and since $\sum_{k=1}^\infty k^{-2n}$ converges, this yields
    \begin{equation*}
        \lim_{N\to\infty}N\sum_{k=N+1}^\infty\frac{1}{k^{2n+2}}=0
        \quad\text{for }n\in\N\;.
    \end{equation*}
    From~(\ref{eq:seriesrepforC1N}) as well as (\ref{eq:evenmoments}) and (\ref{eq:oddmoments}), it follows that, for all $n\in\N_0$,
    \begin{equation*}
        \lim_{N\to\infty}\int_0^1 N C_1^N(t,t) t^n\dd t=
        \lim_{N\to\infty}N\sum_{k=N+1}^\infty\frac{e_{k,n}}{k^2\pi^2}=\frac{1}{(n+1)\pi^2}\;.
    \end{equation*}
    This shows that, for all $n\in\N_0$,
    \begin{equation*}
        \lim_{N\to\infty}\int_0^1 N C_1^N(t,t) t^n\dd t=
        \int_0^1\frac{1}{\pi^2}t^n\dd t\;.
    \end{equation*}
    If the sequence $\{N C_1^N\}_{N\in\N}$ failed to converge pointwise, we could use the Arzel\`{a}--Ascoli theorem and a diagonal argument to construct a second subsequence of $\{N C_1^N\}_{N\in\N}$ converging pointwise but to a different continuous limit function on $(0,1)$ compared to the first subsequence. Since this contradicts the convergence of moments, the claimed result follows.
\end{proof}
We included the on-diagonal convergence in Theorem~\ref{thm:sinflucutation} as a separate statement to demonstrate that Corollary~\ref{cor1} is a consequence of Proposition~\ref{propn:ondiagonal}, which is then used to prove the off-diagonal convergence in Theorem~\ref{thm:sinflucutation}.
\begin{proof}[Proof of Corollary~\ref{cor1}]
    Using the identity that, for $k\in\N$,
    \begin{equation}\label{eq:another_trig_id}
        \cos(2k\pi t)=1-2\left(\sin(k\pi t)\right)^2\;,
    \end{equation}
    we obtain
    \begin{equation*}
        \sum_{k=N+1}^\infty\frac{\cos(2k\pi t)}{k^2\pi^2}
        =\sum_{k=N+1}^\infty\frac{1}{k^2\pi^2}-
         \sum_{k=N+1}^\infty\frac{2\left(\sin(k\pi t)\right)^2}{k^2\pi^2}\;.
    \end{equation*}
    From~(\ref{eq:Nk21}) and Proposition~\ref{propn:ondiagonal}, it follows that, for all $t\in(0,1)$,
    \begin{equation*}
        \lim_{N\to\infty} N\sum_{k=N+1}^\infty\frac{\cos(2k\pi t)}{k^2\pi^2}
        =\frac{1}{\pi^2}-\frac{1}{\pi^2}=0\;,
    \end{equation*}
    as claimed.
\end{proof}
\begin{proof}[Proof of Theorem~\ref{thm:sinflucutation}]
    If $s\in\{0,1\}$ or $t\in\{0,1\}$, the result follows immediately from $\sin(k\pi)=0$ for all $k\in\N_0$, and if $s=t$ for $t\in(0,1)$, the claimed convergence is given by Proposition~\ref{propn:ondiagonal}. Therefore, it remains to consider the off-diagonal case, and we may assume that $s,t\in(0,1)$ are such that $s<t$. Due to the representation~(\ref{eq:mercer4bridge}) and the identity
    \begin{equation*}
        2\sin(k\pi s)\sin(k\pi t)=\cos(k\pi(t-s))-\cos(k\pi(t+s))\;,
    \end{equation*}
    we have
    \begin{align*}
        \min(s,t)-st-\sum_{k=1}^N\frac{2\sin(k\pi s)\sin(k\pi t)}{k^2\pi^2}
        &=\sum_{k=N+1}^\infty \frac{2\sin(k\pi s)\sin(k\pi t)}{k^2\pi^2}\\
        &=\sum_{k=N+1}^\infty \frac{\cos(k\pi(t-s))-\cos(k\pi(t+s))}{k^2\pi^2}\;.
    \end{align*}
    Since $0<t-s<t+s<2$ for $s,t\in(0,1)$ with $s<t$, the convergence away from the diagonal is a consequence of Corollary~\ref{cor1}.
\end{proof}
Note that Theorem~\ref{thm:sinflucutation} states, for $s,t\in[0,1]$,
\begin{equation}\label{eq:ptwsC1}
    \lim_{N\to\infty} N C_1^N(s,t)=
    \begin{cases}
        \frac{1}{\pi^2} & \text{if } s=t\text{ and } t\in(0,1)\\
        0 & \text{otherwise}
    \end{cases}\;,
\end{equation}
which is the key ingredient for obtaining the characterisation of the limit fluctuations for the Karhunen--Lo{\`e}ve expansion given in Theorem~\ref{thm:fluctuations}. We provide the full proof of Theorem~\ref{thm:fluctuations} below after having determined the limit of $2N C_2^N$ as $N\to\infty$.

%% file: fluctuations2.tex
Instead of setting up another moment analysis to study the pointwise limit of $2N C_2^N$ as $N\to\infty$, we simplify the expression for $C_2^N$ from Lemma~\ref{lem:cov2} and deduce the desired pointwise limit from Corollary~\ref{cor1}.

\medskip

Using the standard Fourier basis for $L^2([0,1])$, the polarised Parseval identity and the trigonometric identity~(\ref{eq:trig_identity}), we can write, for $s,t\in[0,1]$,
\begin{align*}
    \min(s,t)&=\int_0^1\ind_{[0,s]}(r)\ind_{[0,t]}(r)\dd r\\
    &=st+\sum_{k=1}^\infty 2\int_0^s \cos(2k\pi r)\dd r\int_0^t \cos(2k\pi r)\dd r
        +\sum_{k=1}^\infty 2\int_0^s \sin(2k\pi r)\dd r\int_0^t \sin(2k\pi r)\dd r\\
    &=st-\sum_{k=1}^\infty\frac{\cos(2k\pi s)}{2k^2\pi^2}
        -\sum_{k=1}^\infty\frac{\cos(2k\pi t)}{2k^2\pi^2}
        +\sum_{k=1}^\infty\frac{\cos(2k\pi (t-s))}{2k^2\pi^2}+\sum_{k=1}^\infty\frac{1}{2k^2\pi^2}\;.
\end{align*}
Applying the identity~(\ref{eq:another_trig_id}) as well as the representation~(\ref{eq:mercer4bridge}) and using the value for $\zeta(2)$ derived in Section~\ref{sect:zeta}, we have
\begin{equation*}
    \sum_{k=1}^\infty\frac{\cos(2k\pi t)}{2k^2\pi^2}=
    \sum_{k=1}^\infty\frac{1}{2k^2\pi^2}-
    \sum_{k=1}^\infty\frac{\left(\sin(k\pi t)\right)^2}{k^2\pi^2}
    =\frac{1}{12}+\frac{t^2-t}{2}\;.
\end{equation*}
Once again exploiting the value for $\zeta(2)$, we obtain
\begin{equation*}
    \min(s,t)-st+\frac{s^2-s}{2}+\frac{t^2-t}{2}+\frac{1}{12}
    =\sum_{k=1}^\infty\frac{\cos(2k\pi (t-s))}{2k^2\pi^2}\;.
\end{equation*}
Using the expression for $C_2^N$ from Lemma~\ref{lem:cov2}, it follows that, for $s,t\in[0,1]$,
\begin{equation*}
    C_2^N(s,t)=\sum_{k=N+1}^\infty\frac{\cos(2k\pi (t-s))}{2k^2\pi^2}\;.
\end{equation*}
This implies that if $t-s$ is an integer then, as a result of the limit~(\ref{eq:Nk21}),
\begin{equation*}
    \lim_{N\to\infty} 2NC_2^N(s,t)=\frac{1}{\pi^2}\;,
\end{equation*}
whereas if $t-s$ is not an integer then, by Corollary~\ref{cor1},
\begin{equation*}
    \lim_{N\to\infty} 2NC_2^N(s,t)=0\;.
\end{equation*}
This can be summarised as, for $s,t\in[0,1]$,
\begin{equation}\label{eq:ptwsC2}
    \lim_{N\to\infty} 2NC_2^N(s,t)=
        \begin{cases}
            \frac{1}{\pi^2} & \text{if } s=t\text{ or } s,t\in\{0,1\}\\\
            0 & \text{otherwise}
        \end{cases}\;.
\end{equation}
We finally prove Theorem~\ref{thm:fluctuations} by considering characteristic functions.
\begin{proof}[Proof of Theorem~\ref{thm:fluctuations}]
    According to Lemma~\ref{lem:cov1} as well as Lemma~\ref{lem:cov2}, the fluctuation processes $(F_t^{N,1})_{t\in[0,1]}$ and $(F_t^{N,2})_{t\in[0,1]}$ are zero-mean Gaussian processes with covariance functions $N C_1^N$ and $2N C_2^N$, respectively.
    
    \medskip
    
    By the pointwise convergences~(\ref{eq:ptwsC1}) and (\ref{eq:ptwsC2}) of the covariance functions in the limit $N\to\infty$, for any $n\in\N$ and any $t_1,\dots,t_n\in[0,1]$, the characteristic functions of the Gaussian random vectors $(F_{t_1}^{N,i},\dots,F_{t_n}^{N,i})$, for $i\in\{1,2\}$, converge pointwise as $N\to\infty$ to the characteristic function of the Gaussian random vector $(F_{t_1}^{i},\dots,F_{t_n}^{i})$. Therefore, the claimed convergences in finite dimensional distributions are consequences of L{\'e}vy's continuity theorem.
\end{proof}

%% file: levyarea.tex
In this section, we consider approximations of second iterated integrals of Brownian motion, which is a classical problem in the numerical analysis of stochastic differential equations (SDEs), see \cite{kloedenplaten}. Due to their presence within stochastic Taylor expansions, increments and second iterated integrals of multidimensional Brownian motion are required by high order strong methods for general SDEs, such as stochastic Taylor \cite{kloedenplaten} and Runge--Kutta \cite{StrongSRK} methods. Currently, the only methodology for exactly generating the increment and second iterated integral, or equivalently the L\'{e}vy area, given by Definition \ref{def:levy_area}, of a $d$-dimensional Brownian motion is limited to the case when $d = 2$. This algorithm for the exact generation of Brownian increments and L\'{e}vy area is detailed in \cite{GainesLyonsInt}.
The approach adapts Marsaglia's ``rectangle-wedge-tail'' algorithm to the joint density function of $\big(W_1^{(1)}, W_1^{(2)}, A_{0,1}^{(1,2)}\big)$, which is expressible as an integral, but can only be evaluated numerically. Due to the subtle relationships between different entries in $A_{0,1}$, it has not been extended to $d>2$. \medskip

Obtaining good approximations of Brownian L\'{e}vy area in an $L^{2}(\mathbb{P})$ sense is known to be difficult. For example, it was shown in \cite{Dickinson} that any approximation of L\'{e}vy area which
is measurable with respect to $N$ Gaussian random variables, obtained from linear functionals of the Brownian path, cannot achieve strong convergence faster than $O(N^{-\frac{1}{2}})$. In particular, this result extends the classical theorem of Clark and Cameron \cite{CameronClark} which establishes a best convergence rate of $O(N^{-\frac{1}{2}})$ for approximations of L\'{e}vy area based on only the Brownian increments $\{W_{(n+1)h} - W_{nh}\}_{0\leq n\leq N - 1}$. Therefore, approximations have been developed which fall outside of this paradigm, see \cite{Davie, Fosterthesis, RecentLevyArea, Wiktorsson}. In the analysis of these methodologies, the L\'{e}vy area of Brownian motion and its approximation are probabilistically coupled in such a way that $L^{2}(\mathbb{P})$ convergence rates of $O(N^{-1})$ can be established.\medskip

We are interested in the approximations of Brownian L\'{e}vy area that can be obtained directly from the Fourier series expansion~(\ref{eq:bridge_fourier_expansion}) and the polynomial expansion~(\ref{eq:bridge_poly_expansion}) of the Brownian bridge.
For the remainder of the section, the Brownian motion $(W_t)_{t\in[0,1]}$ is assumed to be $d$-dimensional and $(B_t)_{t\in[0,1]}$ is its associated Brownian bridge.\medskip

We first recall the standard Fourier approach to the strong approximation of Brownian L\'{e}vy area.

\begin{thm}[Approximation of Brownian L\'{e}vy area via Fourier coefficients, see {\cite[p. 205]{kloedenplaten} and \cite[p. 99]{MilsteinBook}}]\label{thm:levy_fourier_approx}
For $n\in\N$, we define a random antisymmetric $d\times d$ matrix $\widehat{A}_{n}$ by, for $i,j\in\{1,\dots,d\}$,
\begin{equation*}
	 \widehat{A}_{n}^{\m(i,j)} := \frac{1}{2}\left(a_0^{(i)}W_1^{(j)} - W_1^{(i)}a_0^{(j)}\right) + \pi\sum_{k=1}^{n-1} k\left(a_{k}^{(i)}b_k^{(j)} - b_k^{(i)}a_{k}^{(j)}\right)\,,
\end{equation*}
where the normal random vectors $\{a_k\}_{k\in\N_0}$ and $\{b_k\}_{k\in\N}$ are the coefficients from the Brownian bridge expansion (\ref{eq:bridge_fourier_expansion}), that is, the coordinates of each random vector are independent and defined according to (\ref{eq:fourier_coefficients}). Then, for $i,j\in\{1,\dots,d\}$ with $i\neq j$, we have
\begin{align*}
\E\bigg[\Big(A_{0,1}^{(i,j)} - \widehat{A}_{n}^{\m(i,j)}\Big)^2\m\bigg] & = \frac{1}{2\pi^2}\sum_{k = n}^{\infty}\frac{1}{k^2}\;.
\end{align*}
\end{thm}
\begin{remark0}\rm
Using the covariance structure given by (\ref{eq:a02}), (\ref{eq:akal}), (\ref{eq:remaining_covariances}) and the independence of the components of a Brownian bridge, it immediately follows that the coefficients $\{a_k\}_{k\in\N_0}$ and $\{b_k\}_{k\in\N}$ are jointly normal with $a_0\sim \mathcal{N}\big(0,\frac{1}{3}I_d\big)$, $a_k, b_k\sim \mathcal{N}\big(0,\frac{1}{2k^2\pi^2}I_d\big)$, $\cov(a_0, a_k) = -\frac{1}{k^2\pi^2}I_d$ and $\cov(a_l, b_k) = 0$ for $k\in\N$ and $l\in\N_0$.
\end{remark0}
In practice, the above approximation may involve generating the $N$ independent random vectors $\{a_k\}_{1\leq k\leq N}$ followed by the coefficient $a_0$, which will not be independent, but can be expressed as a linear combination of $\{a_k\}_{1\leq k\leq N}$ along with an additional independent normal random vector.
Without this additional normal random vector, we obtain the following discretisation of L\'{e}vy area.
\begin{thm}[Kloeden--Platen--Wright approximation of Brownian L\'{e}vy area, see {\cite{KloedenPlatenWright, MilsteinBook2, Wiktorsson}}]
For $n\in\N$, we define a random antisymmetric $d\times d$ matrix $\widetilde{A}_{n}$ by, for $i,j\in\{1,\dots,d\}$,
\begin{equation*}
	 \widetilde{A}_{n}^{\m(i,j)} := \pi\sum_{k=1}^{n-1} k\left(a_{k}^{(i)}\left(b_k^{(j)} - \frac{1}{k\pi}W_1^{(j)}\right) - \left(b_k^{(i)} - \frac{1}{k\pi}W_1^{(i)}\right)a_{k}^{(j)}\right)\,,
\end{equation*}
where the sequences $\{a_k\}_{k\in\N}$ and $\{b_k\}_{k\in\N}$ of independent normal random vectors are the same as before. Then, for $i,j\in\{1,\dots,d\}$ with $i\neq j$, we have
\begin{align*}
\E\bigg[\Big(A_{0,1}^{(i,j)} - \widetilde{A}_{n}^{\m(i,j)}\Big)^2\m\bigg] & = \frac{3}{2\pi^2}\sum_{k = n}^{\infty}\frac{1}{k^2}\;.
\end{align*}
\end{thm}
\begin{proof}
As for Theorem \ref{thm:levy_fourier_approx}, the result follows by direct calculation. The constant is larger because, for $i\in\{1,\dots,d\}$ and $k\in\N$,
\begin{equation*}
\E\bigg[\Big(b_k^{(i)} - \frac{1}{k\pi}W_1^{(i)}\Big)^2\m\bigg] = \frac{3}{2k^2\pi^2} = 3\,\E\Big[\big(b_k^{(i)}\big)^2\m\Big]\;,
\end{equation*}
which yields the required result.
\end{proof}

Finally, we give the approximation of L\'{e}vy area corresponding to the polynomial expansion~(\ref{eq:bridge_poly_expansion}).
Although this series expansion of Brownian L\'{e}vy area was first proposed in \cite{KuznetsovLevyArea1}, a straightforward derivation based on the polynomial expansion~(\ref{eq:bridge_poly_expansion}) was only established much later in \cite{KuznetsovLevyArea2}. However in \cite{KuznetsovLevyArea1, KuznetsovLevyArea2}, the optimal bound for the
mean squared error of the approximation is not identified.
We will present a similar derivation to \cite{KuznetsovLevyArea2}, but with a simple formula for the mean squared error.

\begin{thm}[Polynomial approximation of Brownian L\'{e}vy area, see {\cite[p. 47]{KuznetsovLevyArea1}} and {\cite{KuznetsovLevyArea2}}]\label{thm:poly_area_error}
For $n\in\N_0$, we define a random antisymmetric $d\times d$ matrix $\widebar{A}_{n}$ by, for $n\in\N$ and $i,j\in\{1,\dots,d\}$,
\begin{equation*}
\widebar{A}_{n}^{\m(i,j)} := \frac{1}{2}\left(W_1^{(i)}c_1^{(j)} - c_1^{(i)}W_1^{(j)}\right) + \frac{1}{2}\sum_{k=1}^{n-1}\left(c_k^{(i)}c_{k+1}^{(j)} - c_{k+1}^{(i)}c_k^{(j)}\right)\,,
\end{equation*}
where the normal random vectors $\{c_k\}_{k\in\N}$ are the coefficients from the polynomial expansion (\ref{eq:bridge_poly_expansion}), that is, the coordinates are independent and defined according to (\ref{eq:poly_coefficients}), and we set
\begin{equation*}
    \widebar{A}_{0}^{\m(i,j)} := 0\,.
\end{equation*}
Then, for $n\in\N_0$ and for $i,j\in\{1,\dots,d\}$ with $i\neq j$, we have
\begin{equation*}
\E\bigg[\Big(A_{0,1}^{(i,j)} - \widebar{A}_{n}^{\m(i,j)}\Big)^2\m\bigg] = \frac{1}{8n+4}\;.
\end{equation*}
\end{thm}

\begin{remark0}\label{remark:poly}\rm
By applying Lemma \ref{lem:itoBB}, the orthogonality of shifted Legendre polynomials and the independence of the components of a Brownian bridge, we see that the coefficients $\{c_k\}_{k\in\N}$ are independent and distributed as $c_k\sim\mathcal{N}\big(0,\frac{1}{2k + 1}I_d\big)$ for $k\in\N$.
\end{remark0}

\begin{proof} It follows from the polynomial expansion (\ref{eq:bridge_poly_expansion}) that, for $i,j\in\{1,\dots,d\}$ with $i\neq j$,
\begin{equation}\label{eq:initial_expand}
\int_0^1 B_t^{(i)}\dd B_t^{(j)} = \int_0^1 \left(\sum_{k=1}^\infty (2k+1)\, c_k^{(i)} \int_0^t Q_k(r)\dd r\right)\dd\left(\sum_{l=1}^\infty (2l+1)\, c_l^{(j)} \int_0^t Q_l(r)\dd r \right)\,,
\end{equation}
where the series converge in $L^2(\mathbb{P})$. To simplify~(\ref{eq:initial_expand}), we use the identities in~(\ref{eq:id_shiftedLP}) for shifted Legendre polynomials
as well as the orthogonality of shifted Legendre polynomials to obtain that, for $k, l\in\N$,
\begin{align*}
\int_0^1 \left(\int_0^t Q_k(r)\dd r\right) \dd \left(\int_0^t Q_l(r)\dd r\right) & = \int_0^1 Q_l(t)\int_0^t Q_k(r)\dd r \dd t\\
& = \dfrac{1}{2(2k+1)}\m{\displaystyle\int_0^1 Q_l(t)\left(Q_{k+1}(t) - Q_{k-1}(t)\right) \dd t}\\[3pt]
& = \begin{cases}\phantom{-}\dfrac{1}{2(2k+1)}\m{\displaystyle\int_0^1 \left(Q_{k+1}(t)\right)^2 \dd t} & \text{if }l = k + 1\\[9pt]
-\dfrac{1}{2(2k+1)}\m{\displaystyle\int_0^1 \left(Q_{k-1}(t)\right)^2 \dd t} & \text{if }l = k - 1\\[8pt]
\phantom{-}0 & \text{otherwise}\end{cases}\;.
\end{align*}

Evaluating the above integrals gives, for $k, l\in\N$,
\begin{align}\label{eq:relation4legendre}
\int_0^1 \left(\int_0^t Q_k(r)\dd r\right) \dd \left(\int_0^t Q_l(r)\dd r\right) & = \begin{cases}
\phantom{-}\dfrac{1}{2(2k+1)(2k + 3)} & \text{if }l = k + 1\\[9pt]
-\dfrac{1}{2(2k+1)(2k - 1)} & \text{if }l = k - 1\\[8pt]
\phantom{-}0 & \text{otherwise}\end{cases}\;.
\end{align}
In particular, for $k,l\in\N$, this implies that
\begin{align*}
\int_0^1 \left((2k+1) c_k^{(i)} \int_0^t Q_k(r)\dd r\right) \dd \left((2l+1) c_l^{(j)} \int_0^t Q_l(r)\dd r\right) & = \begin{cases}
\phantom{-}\dfrac{1}{2}c_k^{(i)}c_{k+1}^{(j)} & \text{if }l = k + 1\\[9pt]
-\dfrac{1}{2}c_k^{(i)}c_{k-1}^{(j)} & \text{if }l = k - 1\\[7pt]
\phantom{-}0 & \text{otherwise}\end{cases}\;.
\end{align*}
Therefore, by the bounded convergence theorem in $L^2(\mathbb{P})$, we can simplify the expansion~(\ref{eq:initial_expand}) to
\begin{align}\label{eq:second_expand}
\int_0^1 B_t^{(i)}\dd B_t^{(j)} & = \frac{1}{2}\sum_{k=1}^{\infty}\left(c_k^{(i)}c_{k+1}^{(j)} - c_{k+1}^{(i)}c_k^{(j)}\right)\;,
\end{align}
where, just as before, the series converges in $L^2(\mathbb{P})$. Since $W_t = t\m W_1 + B_t$ for $t\in[0,1]$, we have, for $i,j\in\{1,\dots,d\}$ with $i\neq j$,
\begin{align*}
\int_0^1 W_t^{(i)}\dd W_t^{(j)} & =  \int_0^1 \big(t W_1^{(i)}\big)\dd \big(t W_1^{(j)}\big) + \int_0^1 B_t^{(i)}\dd \big(t W_1^{(j)}\big)+ \int_0^1 \big(t W_1^{(i)}\big)\dd B_t^{(j)} + \int_0^1 B_t^{(i)}\dd B_t^{(j)}\\[3pt]
& = \frac{1}{2}W_1^{(i)}W_1^{(j)} - W_1^{(j)}\int_0^1 t \dd B_t^{(i)} + W_1^{(i)}\int_0^1 t \dd  B_t^{(j)} + \int_0^1 B_t^{(i)}\dd B_t^{(j)}\;,
\end{align*}
where the second line follows by integration by parts. As $$\int_0^1 W_t^{(i)}\dd W_t^{(j)} = \frac{1}{2}W_1^{(i)}W_1^{(j)} + A_{0,1}^{(i,j)}$$ and $Q_1(t) = 2t - 1$, the above and (\ref{eq:second_expand}) imply that, for $i,j\in\{1,\dots,d\}$,
\begin{align*}
A_{0,1}^{(i,j)} = \frac{1}{2}\left(W_1^{(i)}c_1^{(j)} - c_1^{(i)}W_1^{(j)}\right) + \frac{1}{2}\sum_{k=1}^{\infty}\left(c_k^{(i)}c_{k+1}^{(j)} - c_{k+1}^{(i)}c_k^{(j)}\right)\;.
\end{align*}
By the independence of the normal random vectors in the sequence $\{c_k\}_{k\in\N}$, it is straightforward to compute the mean squared error in approximating $A_{0,1}$ and we obtain, for $n\in\N$ and for $i,j\in\{1,\dots,d\}$ with $i\neq j$,
\begin{align*}
\E\bigg[\Big(A_{0,1}^{(i,j)} - \widebar{A}_{n}^{\m(i,j)}\Big)^2\m\bigg] & = \E\left[\left(\frac{1}{2}\sum_{k=n}^{\infty} \left(c_{k}^{(i)}c_{k+1}^{(j)} - c_{k+1}^{(i)}c_{k}^{(j)}\right)\right)^2\m\right]\\
& = \frac{1}{4}\sum_{k=n}^{\infty} \frac{2}{(2k+1)(2k+3)}\\
& = \frac{1}{4}\sum_{k=n}^{\infty} \left(\frac{1}{2k+1} - \frac{1}{2k+3}\right)\\
& = \frac{1}{8n+4}\;,
\end{align*}
by Remark \ref{remark:poly}. Similarly, as the normal random vector $W_1$ and the ones in the sequence $\{c_k\}_{k\in\N}$ are independent, we have
\begin{align*}
\E\bigg[\Big(A_{0,1}^{(i,j)} - \widebar{A}_{0}^{(i,j)}\Big)^2\m\bigg] & =\E\left[\left(\frac{1}{2}\left(W_1^{(i)}c_1^{(j)} - c_1^{(i)}W_1^{(j)}\right)\right)^2\m\right] + \E\left[\left(\frac{1}{2}\sum_{k=1}^{\infty} \left(c_{k}^{(i)}c_{k+1}^{(j)} - c_{k+1}^{(i)}c_{k}^{(j)}\right)\right)^2\m\right]\\
& = \frac{1}{6} + \frac{1}{12} = \frac{1}{4}\;,
\end{align*}
as claimed.
\end{proof}

Given that we have now considered three different strong approximations of Brownian L\'{e}vy area,
it is reasonable to compare their respective rates of convergence. Combining the above theorems, we obtain the following result.

\begin{cor}[Asymptotic convergence rates of L\'{e}vy area approximations]
For $n\in\N$, we set $N = 2n$ so that the number of Gaussian random vectors required to define the L\'{e}vy area approximations $\widehat{A}_{n}, \widetilde{A}_{n}$ and $\widebar{A}_{2n}$ is $N$ or $N-1$, respectively.
Then, for $i,j\in\{1,\dots,d\}$ with $i\neq j$ and as $N\to\infty$, we have
\begin{align*}
\E\bigg[\Big(A_{0,1}^{(i,j)} - \widehat{A}_{n}^{\m(i,j)}\Big)^2\m\bigg] & \sim \frac{1}{\pi^2}\bigg(\frac{1}{N}\bigg)\;,\\[3pt]
\E\bigg[\Big(A_{0,1}^{(i,j)} - \widetilde{A}_{n}^{\m(i,j)}\Big)^2\m\bigg] & \sim \frac{3}{\pi^2}\bigg(\frac{1}{N}\bigg)\;,\\[3pt]
\E\bigg[\Big(A_{0,1}^{(i,j)} - \widebar{A}_{2n}^{\m(i,j)}\Big)^2\m\bigg] & \sim \frac{1}{8}\bigg(\frac{1}{N}\bigg)\;.
\end{align*}
In particular, the polynomial approximation of Brownian L\'{e}vy area is more accurate than the Kloeden--Platen--Wright approximation, both of which use only independent Gaussian vectors.
\end{cor}
\begin{remark0}\rm
It was shown in \cite{Dickinson} that $\frac{1}{\pi^2}\hspace{-0.5mm}\left(\frac{1}{N}\right)$ is the optimal asymptotic rate of mean squared convergence for L\'{e}vy area approximations that are measurable with respect to $N$ Gaussian random variables, obtained from linear functionals of the Brownian path.
\end{remark0}

As one would expect, all the L\'{e}vy area approximations converge in $L^{2}(\mathbb{P})$ with a rate of $O(N^{-\frac{1}{2}})$
and thus the main difference between their respective accuracies is in the leading error constant.
More concretely, for sufficiently large $N$, the approximation based on the Fourier expansion of the Brownian bridge is roughly 11\% more accurate in $L^{2}(\mathbb{P})$ than that of the polynomial approximation.
On the other hand, the polynomial approximation is easier to implement in practice as all of the required coefficients are independent. Since it has the largest asymptotic error constant, the Kloeden--Platen--Wright approach gives the least accurate approximation for Brownian L\'{e}vy area.\medskip

We observe that the leading error constants for the L\'{e}vy area approximations resulting from the Fourier series and the polynomial expansion coincide with the average $L^2(\mathbb{P})$ error of their respective fluctuation processes, that is, applying Fubini's theorem followed by the limit theorems for the fluctuation processes $(F_t^{N,2})_{t\in[0,1]}$ and $(F_t^{N,3})_{t\in[0,1]}$ defined by~(\ref{fourier_fluctuation}) and (\ref{poly_fluctuation}), respectively, gives
\begin{align*}
    \lim_{N\to\infty}\E\left[\int_0^1 \left(F_t^{N,2}\right)^2 \dd t \right]
    &=\int_0^1\frac{1}{\pi^2}\dd t=\frac{1}{\pi^2}\;,\\[3pt]
    \lim_{N\to\infty}\E\left[\int_0^1 \left(F_t^{N,3}\right)^2 \dd t \right]
    &=\int_0^1 \frac{1}{\pi}\sqrt{t(1-t)}\dd t=\frac{1}{8}\;.
\end{align*}
To demonstrate how this correspondence arises, we close with some heuristics. For $N\in\N$,
we consider an approximation of the Brownian bridge which uses $N$ random vectors, and we denote the corresponding approximation of Brownian motion $(W_t)_{t\in[0,1]}$ by $(S_t^N)_{t\in[0,1]}$, where the difference between Brownian motion and its associated Brownian bridge is the first term in the approximation. In the Fourier and polynomial approaches, the error in approximating Brownian L\'{e}vy area is then essentially given by
\begin{equation*}
    \int_0^1 W_t^{(i)}\dd W_t^{(j)}-\int_0^1 S_t^{N,(i)}\dd S_t^{N,(j)}
    =\int_0^1 \left(W_t^{(i)}-S_t^{N,(i)}\right)\dd W_t^{(j)}
    +\int_0^1 S_t^{N,(i)}\dd \left(W_t^{(j)}-S_t^{N,(j)}\right)\;.
\end{equation*}
If one can argue that
\begin{equation*}
    \int_0^1 S_t^{N,(i)}\dd \left(W_t^{(j)}-S_t^{N,(j)}\right)=O\left(\frac{1}{N}\right)\;,
\end{equation*}
which, for instance, for the polynomial approximation follows directly from~(\ref{eq:relation4legendre}) and Remark~\ref{remark:poly}, then in terms of the fluctuation processes $(F_t^N)_{t\in[0,1]}$ defined by
\begin{equation*}
    F_t^N=\sqrt{N}\left(W_t-S_t^{N}\right)\;,
\end{equation*}
the error of the L\'{e}vy area approximation can be expressed as
\begin{equation*}
    \frac{1}{\sqrt{N}}\int_0^1 F_t^{N,(i)}\dd W_t^{(j)}+O\left(\frac{1}{N}\right)\;.
\end{equation*}
Thus, by It\^{o}'s isometry and Fubini's theorem, the leading error constant in the mean squared error is indeed given by
\begin{equation*}
    \int_0^1\lim_{N\to\infty}\E\left[\left(F_t^{N,(i)}\right)^2\m\right]\dd t\;.
\end{equation*}\smallskip

This connection could be interpreted as an asymptotic It\^{o} isometry for L\'{e}vy area approximations.

%% file: appendix.tex
\begin{center}
\begin{table}[h]
  \centering
  \begin{tabular}{cl}
    \toprule\\[-9pt]
    \begin{minipage}{4.5cm} \begin{center} Type of expansion \end{center}\end{minipage} & \begin{minipage}{9.5cm} \begin{center}Expansion of the Brownian bridge $(B_t)_{t\in[0,1]}$\end{center}\end{minipage}\\
     & \\[-9pt]
    \midrule
     \begin{minipage}{4.5cm}
     \begin{center}
     Karhunen--Lo\`{e}ve\\[3pt] 
(Lo{\`e}ve~\cite{loeve2})
     \end{center}
     \end{minipage} & 
     \begin{minipage}{9.5cm}
     \begin{equation*}
     B_t=\sum_{k=1}^\infty\frac{2\sin(k\pi t)}{k\pi}\int_0^1\cos(k\pi r)\dd B_r
     \end{equation*}     
     \end{minipage}\\
     & \\[-9pt]
     \midrule 
     \begin{minipage}{4.5cm}
     \begin{center}
     Fourier series\\[3pt] 
(Kahane~\cite{kahane} or Kloeden--Platen~\cite{kloedenplaten})
     \end{center}
     \end{minipage} &     
     \begin{minipage}{9.5cm}
     \begin{equation*}
     B_t=\frac{1}{2}a_0+\sum_{k=1}^\infty\left(
     a_k\cos(2k\pi t)+b_k\sin(2k\pi t)\right)
     \end{equation*}\\
     with, for $k\in\N_0$,\smallskip 
    \begin{align*}
      a_k=2\int_0^1 \cos(2k\pi r)B_r\dd r\;,\quad
      b_k=2\int_0^1 \sin(2k\pi r)B_r\dd r
    \end{align*}
     \end{minipage}\\
     & \\[-9pt] 
     \midrule
     \begin{minipage}{4.5cm}
     \begin{center}
     Polynomial \\[3pt] 
     (Foster, Lyons, Oberhauser~\cite{foster} and Habermann~\cite{semicircle})
     \end{center}
     \end{minipage} & \begin{minipage}{9.5cm}
     \begin{equation*}
     B_t=\sum_{k=1}^\infty (2k+1)\m c_k \int_0^t Q_k(r)\dd r
     \end{equation*}\\
     with, for $k\in\N$, 
     \begin{equation*}
     c_k= \int_0^1 Q_k(r)\dd B_r
     \end{equation*}\\
     and $Q_k$ denoting the shifted Legendre polynomial of degree $k$
     \end{minipage}\\
     & \\[-9pt]
    \bottomrule
  \end{tabular}
  \caption{Table summarising the Brownian bridge expansions considered in this paper.}
  \label{table:bridge_expansions}
\end{table}

\newpage
\begin{table}[ht]
  \centering
  \begin{tabular}{cl}
    \toprule\\[-9pt]
    \begin{minipage}{4.5cm} \begin{center}Type of expansion\end{center}\end{minipage} & \begin{minipage}{9.5cm} \begin{center}Expansion of the Brownian L\'{e}vy area $A_{0,1}$\end{center}\end{minipage}\\
     & \\[-9pt]
    \midrule
     \begin{minipage}{4.5cm}
     \begin{center}
     Fourier series \\[3pt] 
     (Kloeden--Platen \cite{kloedenplaten} and Milstein \cite{MilsteinBook})
     \end{center}
     \end{minipage} &     
     \begin{minipage}{9.5cm}
     \begin{equation*}
	 A_{0,1}^{(i,j)} = \frac{1}{2}\left(a_0^{(i)}W_1^{(j)} - W_1^{(i)}a_0^{(j)}\right) + \pi\sum_{k=1}^\infty k\left(a_{k}^{(i)}b_k^{(j)} - b_k^{(i)}a_{k}^{(j)}\right)
     \end{equation*}\\
     with, for $k\in\N_0$,\smallskip 
    \begin{align*}
      a_k^{(i)}=2\int_0^1 \cos(2k\pi r)B_r^{(i)}\dd r\;,\quad
      b_k^{(i)}=2\int_0^1 \sin(2k\pi r)B_r^{(i)}\dd r
    \end{align*}
     \end{minipage}\\
     & \\[-9pt] 
     \midrule\\[-9pt]
     \begin{minipage}{4.5cm}
     \begin{center}
     Fourier series\\[3pt]
     (Kloeden--Platen--Wright \cite{KloedenPlatenWright} and Milstein \cite{MilsteinBook2})\\[3pt] 
     \end{center}
     \end{minipage} & \begin{minipage}{9.5cm}
     \begin{equation*}
	 A_{0,1}^{(i,j)} = \pi\sum_{k=1}^\infty k\left(a_{k}^{(i)}\left(b_k^{(j)} - \frac{W_1^{(j)}}{k\pi}\right) - \left(b_k^{(i)} - \frac{W_1^{(i)}}{k\pi}\right)a_{k}^{(j)}\right)
     \end{equation*}
     \end{minipage}\\
     & \\[-9pt]
     \midrule
     \begin{minipage}{4.5cm}
     \begin{center}
     Polynomial \\[3pt] 
     (Kuznetsov \cite{KuznetsovLevyArea1})
     \end{center}
     \end{minipage} & \begin{minipage}{9.5cm}
     \begin{equation*}
	 A_{0,1}^{(i,j)} = \frac{1}{2}\left(W_1^{(i)}c_1^{(j)} - c_1^{(i)}W_1^{(j)}\right) + \frac{1}{2}\sum_{k=1}^\infty \left(c_k^{(i)}c_{k+1}^{(j)} - c_{k+1}^{(i)}c_k^{(j)}\right)
     \end{equation*}\\
     with, for $k\in\N$,\smallskip 
     \begin{equation*}
     c_k^{(i)}= \int_0^1 Q_k(r)\dd B_r^{(i)}
     \end{equation*}\\
     and $Q_k$ denoting the shifted Legendre polynomial of degree $k$
     \end{minipage}\\
     & \\[-9pt]
    \bottomrule
  \end{tabular}
  \caption{Table summarising the L\'{e}vy area expansions considered in this paper.}
  \label{table:area_expansions}
  \end{table}
\end{center}

%% file: main.bbl
\begin{thebibliography}{10}

\bibitem{ArfkenWeber}
George~B. Arfken and Hans~J. Weber.
\newblock {\em Mathematical Methods for Physicists}.
\newblock Elsevier, sixth edition, 2005.

\bibitem{WeakMLMC}
Denis Belomestny and Tigran Nagapetyan.
\newblock {Multilevel path simulation for weak approximation schemes with
  application to L\'{e}vy-driven SDEs}.
\newblock {\em Bernoulli}, 23(2):927--950, 2017.

\bibitem{zeta_bernoulli}
Zenon~I. Borevich and Igor~R. Shafarevich.
\newblock {\em Number {T}heory}.
\newblock Translated from the Russian by Newcomb Greenleaf. Pure and Applied
  Mathematics, Volume 20. Academic Press, New York, 1966.

\bibitem{CameronClark}
John M.~C. Clark and R.~J. Cameron.
\newblock The maximum rate of convergence of discrete approximations for
  stochastic differential equations.
\newblock In {\em Stochastic Differential Systems Filtering and Control}.
  Springer, Berlin, 1980.

\bibitem{Davie}
Alexander Davie.
\newblock K{MT} theory applied to approximations of {SDE}.
\newblock In {\em Stochastic Analysis and Applications}, volume 100 of {\em
  Springer Proceedings in Mathematics and Statistics}, pages 185--201.
  Springer, Cham, 2014.

\bibitem{WeakAntitheticMLMC}
Kristian Debrabant, Azadeh Ghasemifard, and Nicky~C. Mattsson.
\newblock {Weak Antithetic MLMC Estimation of SDEs with the Milstein scheme for
  Low-Dimensional Wiener Processes}.
\newblock {\em Applied Mathematics Letters}, 91(4):22--27, 2019.

\bibitem{HighOrderMLMC}
Kristian Debrabant and Andreas R{\"{o}}{\ss}ler.
\newblock {On the acceleration of the multi-level Monte Carlo method}.
\newblock {\em Journal of Applied Probability}, 52(2):307--322, 2015.

\bibitem{Dickinson}
Andrew~S. Dickinson.
\newblock {Optimal Approximation of the Second Iterated Integral of Brownian
  Motion}.
\newblock {\em Stochastic Analysis and Applications}, 25(5):1109--1128, 2007.

\bibitem{chebfun2}
Silviu Filip, Aurya Javeed, and Lloyd~N. Trefethen.
\newblock Smooth random functions, random {ODE}s, and {G}aussian processes.
\newblock {\em SIAM Review}, 61(1):185--205, 2019.

\bibitem{Flint}
Guy Flint and Terry Lyons.
\newblock {Pathwise approximation of {SDE}s by coupling piecewise abelian rough
  paths}.
\newblock {\em \\\rm
  \href{https://arxiv.org/abs/1505.01298}{https://arxiv.org/abs/1505.01298}},
  2015.

\bibitem{Fosterthesis}
James Foster.
\newblock {\em Numerical approximations for stochastic differential equations}.
\newblock PhD thesis, University of Oxford, 2020.

\bibitem{foster}
James Foster, Terry Lyons, and Harald Oberhauser.
\newblock An optimal polynomial approximation of {B}rownian motion.
\newblock {\em SIAM Journal on Numerical Analysis}, 58(3):1393--1421, 2020.

\bibitem{GainesLyonsInt}
Jessica Gaines and Terry Lyons.
\newblock {Random Generation of Stochastic Area Integrals}.
\newblock {\em SIAM Journal on Applied Mathematics}, 54(4):1132--1146, 1994.

\bibitem{GainesLyonsVar}
Jessica Gaines and Terry Lyons.
\newblock {Variable step size control for stochastic differential equations}.
\newblock {\em SIAM Journal on Applied Mathematics}, 57(5):1455--1484, 1997.

\bibitem{MilsteinMLMC}
Michael~B. Giles.
\newblock Improved multilevel {M}onte {C}arlo convergence using the {M}ilstein
  scheme.
\newblock In {\em Monte {C}arlo and quasi-{M}onte {C}arlo methods 2006}, pages
  343--358. Springer, Berlin, 2008.

\bibitem{MLMC}
Michael~B. Giles.
\newblock {Multilevel Monte Carlo path simulation}.
\newblock {\em Operations Research}, 56(3):607--617, 2008.

\bibitem{AntitheticMLMC}
Michael~B. Giles and Lukasz Szpruch.
\newblock {Antithetic multilevel Monte Carlo estimation for multi-dimensional
  SDEs without L\'{e}vy area simulation}.
\newblock {\em Annals of Applied Probability}, 24(4):1585--1620, 2014.

\bibitem{HabermannRecent}
Karen Habermann.
\newblock {Asymptotic error in the eigenfunction expansion for the Green's
  function of a Sturm--Liouville problem}.
\newblock {\em \rm
  \href{https://arxiv.org/abs/2109.10887}{https://arxiv.org/abs/2109.10887}},
  2021.

\bibitem{semicircle}
Karen Habermann.
\newblock A semicircle law and decorrelation phenomena for iterated
  {K}olmogorov loops.
\newblock {\em Journal of the London Mathematical Society}, 103(2):558--586,
  2021.

\bibitem{iwaniec}
Henryk Iwaniec.
\newblock {\em Topics in {C}lassical {A}utomorphic {F}orms}, volume~17 of {\em
  Graduate Studies in Mathematics}.
\newblock American Mathematical Society, Providence, 1997.

\bibitem{kahane}
Jean-Pierre Kahane.
\newblock {\em Some {R}andom {S}eries of {F}unctions}, volume~5 of {\em
  Cambridge Studies in Advanced Mathematics}.
\newblock Cambridge University Press, second edition, 1985.

\bibitem{kloedenplaten}
Peter~E. Kloeden and Eckhard Platen.
\newblock {\em Numerical {S}olution of {S}tochastic {D}ifferential
  {E}quations}, volume~23 of {\em Applications of Mathematics}.
\newblock Springer, Berlin, 1992.

\bibitem{KloedenPlatenWright}
Peter~E. Kloeden, Eckhard Platen, and Ian~W. Wright.
\newblock {The approximation of multiple stochastic integrals}.
\newblock {\em Stochastic Analysis and Applications}, 10(4):431--441, 1992.

\bibitem{KuznetsovLevyArea1}
Dmitriy~F. Kuznetsov.
\newblock {A method of expansion and approximation of repeated stochastic
  {S}tratonovich integrals based on multiple {F}ourier series on full
  orthonormal systems. [In Russian]}.
\newblock {\em Electronic Journal ``Differential Equations and Control
  Processes''}, 1:18--77, 1997.

\bibitem{KuznetsovLevyArea2}
Dmitriy~F. Kuznetsov.
\newblock {New Simple Method of Expansion of Iterated {I}to Stochastic
  integrals of Multiplicity 2 Based on Expansion of the {B}rownian Motion Using
  {L}egendre Polynomials and Trigonometric Functions}.
\newblock {\em \\\rm
  \href{https://arxiv.org/abs/1807.00409}{https://arxiv.org/abs/1807.00409}},
  2019.

\bibitem{SRKforMCMC}
Xuechen Li, Denny Wu, Lester Mackey, and Murat~A. Erdogdu.
\newblock {Stochastic Runge-Kutta Accelerates Langevin Monte Carlo and Beyond}.
\newblock {\em Advances in Neural Information Processing Systems}, 2019.

\bibitem{loeve2}
Michel Lo\`eve.
\newblock {\em Probability theory {II}}, volume~46 of {\em Graduate Texts in
  Mathematics}.
\newblock Springer, fourth edition, 1978.

\bibitem{basel}
Pietro Mengoli.
\newblock {\em Novae quadraturae arithmeticae, seu de Additione fractionum}.
\newblock Ex Typographia Iacobi Montij, 1650.

\bibitem{mercer}
James Mercer.
\newblock {XVI}. {F}unctions of {P}ositive and {N}egative {T}ype, and their
  {C}onnection with the {T}heory of {I}ntegral {E}quations.
\newblock {\em Philosophical Transactions of the Royal Society of London.
  Series A}, 209:415--446, 1909.

\bibitem{MilsteinBook2}
Grigori~N. Milstein.
\newblock {\em Numerical Integration of Stochastic Differential Equations}.
\newblock [In Russian]. Ural University Press, Sverdlovsk, 1988.

\bibitem{MilsteinBook}
Grigori~N. Milstein.
\newblock {\em Numerical Integration of Stochastic Differential Equations},
  volume 313.
\newblock Springer Science {\&} Business Media, 1994.

\bibitem{RecentLevyArea}
Jan Mrongowius and Andreas R\"{o}\ss{}ler.
\newblock On the approximation and simulation of iterated stochastic integrals
  and the corresponding {L}\'{e}vy areas in terms of a multidimensional
  {B}rownian motion.
\newblock {\em Stochastic Analysis and Applications}, 40(3):397--425, 2022.

\bibitem{StrongSRK}
Andreas R{\"{o}}{\ss}ler.
\newblock {Runge--Kutta Methods for the Strong Approximation of Solutions of
  Stochastic Differential Equations}.
\newblock {\em SIAM Journal on Numerical Analysis}, 8(3):922--952, 2010.

\bibitem{chebfun1}
Nick Trefethen.
\newblock {\it Brownian paths and random polynomials}, Version June 2019.
  Chebfun Example.
\newblock \\
  \href{https://www.chebfun.org/examples/stats/RandomPolynomials.html}{https://www.chebfun.org/examples/stats/RandomPolynomials.html}.

\bibitem{Wiktorsson}
Magnus Wiktorsson.
\newblock {Joint characteristic function and simultaneous simulation of
  iterated Itô integrals for multiple independent Brownian motions}.
\newblock {\em Annals of Applied Probability}, 11(2):470--487, 2001.

\end{thebibliography}
